\documentclass[10pt]{amsproc}
\usepackage[utf8]{inputenc}
\usepackage[dvipsnames]{xcolor}
\usepackage{todonotes}

\usepackage{etoolbox}
\usepackage{rotating}

\makeatletter
\patchcmd{\@thm}{\let\thm@indent\indent}{\let\thm@indent\noindent}{}{}
\patchcmd{\@thm}{\thm@headfont{\scshape}}{\thm@headfont{\bfseries}}{}{}

\usepackage[T1]{fontenc}

\usepackage{hyperref}

\usepackage{mathtools}

\let\G\undefined

\usepackage{multirow}
\usepackage{longtable}

\usepackage{amsmath}
\usepackage{amsfonts}
\usepackage{amssymb}
\usepackage{float}
\usepackage{amsthm}
\usepackage{comment}
\usepackage{amscd}
\usepackage{amsxtra}
\usepackage{epsfig}
\usepackage{epigraph}
\usepackage{pgfplots}
\usepackage{pdflscape} 
\usepackage{afterpage}

\usepackage{stmaryrd}

\usepackage{listings}
\usepackage{color}

\definecolor{dkgreen}{rgb}{0,0.6,0}
\definecolor{gray}{rgb}{0.5,0.5,0.5}
\definecolor{mauve}{rgb}{0.58,0,0.82}

\usepackage{color}
\usepackage[all]{xy}
\usepackage{verbatim}

\usepackage{eucal}
\usepackage{mathrsfs}

\usepackage{graphicx}
\usepackage{tikz-cd}

\usepackage{url}
\usepackage{verbatim}

\usepackage{xy}
\xyoption{all}
\newcommand{\xar}[1]{\xrightarrow{{#1}}}
\usepackage{adjustbox}

\usepackage{amsthm}
\usepackage{hhline}
\usepackage{enumitem}

\usepackage{caption} 
\makeatletter
\def\@tocline#1#2#3#4#5#6#7{\relax
  \ifnum #1>\c@tocdepth 
  \else
    \par \addpenalty\@secpenalty\addvspace{#2}%
    \begingroup \hyphenpenalty\@M
    \@ifempty{#4}{%
      \@tempdima\csname r@tocindent\number#1\endcsname\relax
    }{%
      \@tempdima#4\relax
    }%
    \parindent\z@ \leftskip#3\relax \advance\leftskip\@tempdima\relax
    \rightskip\@pnumwidth plus4em \parfillskip-\@pnumwidth
    #5\leavevmode\hskip-\@tempdima
      \ifcase #1
       \or\or \hskip 1em \or \hskip 2em \else \hskip 3em \fi%
      #6\nobreak\relax
    \hfill\hbox to\@pnumwidth{\@tocpagenum{#7}}\par
    \nobreak
    \endgroup
  \fi}
\makeatother

\usepackage[noabbrev,capitalize]{cleveref} 

\crefformat{nul}{(#2#1#3)}
\Crefformat{nul}{(#2#1#3)}

\crefname{section}{\S}{\S\S}
\crefname{subsection}{\S}{\S\S}
\crefname{axioms}{Axiom}{Axioms}
\crefname{exercise}{Exercise}{Exercises}
\crefname{exercisenum}{Exercise}{Exercises}
\crefname{construction}{Construction}{Constructions}
\crefname{problem}{Problem}{Problems}
\crefname{theorem}{Theorem}{Theorems}
\crefname{definition}{Definition}{Definitions}
\crefname{prop}{Proposition}{Propositions}
\crefname{lemma}{Lemma}{Lemmas}
\crefname{example}{Example}{Examples}
\crefname{examplealph}{Example}{Examples}
\crefname{corollary}{Corollary}{Corollaries}
\crefname{nonexample}{Nonexample}{Nonexamples}
\crefname{equation}{}{}
\crefname{summary}{Summary}{Summaries}
\crefname{recollection}{Recollection}{Recollections}
\Crefname{recollection}{Recollection}{Recollections}
\Crefname{nonexample}{Nonexample}{Nonexamples}
\Crefname{corollary}{Corollary}{Corollaries}
\Crefname{corollary}{Corollary}{Corollaries}
\Crefname{axioms}{Axiom}{Axioms}
\Crefname{exercise}{Exercise}{Exercises}
\Crefname{exercisenum}{Exercise}{Exercises}
\Crefname{construction}{Construction}{Constructions}
\Crefname{problem}{Problem}{Problems}
\Crefname{theorem}{Theorem}{Theorems}
\Crefname{definition}{Definition}{Definitions}
\Crefname{prop}{Proposition}{Propositions}
\Crefname{lemma}{Lemma}{Lemmas}
\Crefname{example}{Example}{Examples}
\Crefname{examplealph}{Example}{Examples}
\Crefname{section}{\S}{\S\S}
\Crefname{subsection}{\S}{\S\S}

\DeclareMathOperator{\Hom}{Hom}

\DeclareMathOperator{\spec}{Spec}
\DeclareMathOperator{\spf}{Spf}

\DeclareMathOperator{\Tr}{Tr}

\DeclareMathOperator{\Sym}{Sym}

\newtheorem{lemma}{Lemma}[section]

\newtheorem{theorem}[lemma]{Theorem}

\newtheorem{prop}[lemma]{Proposition}
\newtheorem{conjecture}[lemma]{Conjecture}

\newtheorem*{conj-moore}{Conjecture~\ref{moore-splitting}}

\theoremstyle{definition}

\newtheorem{definition}[lemma]{Definition}
\newtheorem{warning}[lemma]{Warning}
\newtheorem{example}[lemma]{Example}

\newtheorem{construction}[lemma]{Construction}
\newtheorem{remark}[lemma]{Remark}
\newtheorem{notation}[lemma]{Notation}

\newtheorem{observe}[lemma]{Observation}

\renewcommand{\AA}{\mathbf{A}}
\newcommand{\FF}{\mathbf{F}}
\newcommand{\Z}{\mathbf{Z}}
\newcommand{\QQ}{\mathbf{Q}}
\newcommand{\cc}{\mathbf{C}}

\newcommand{\cC}{\mathcal{C}}
\newcommand{\cd}{\mathcal{D}}

\newcommand{\Sp}{\mathrm{Sp}}

\newcommand{\co}{\mathcal{O}}
\newcommand{\ce}{\mathcal{E}}
\newcommand{\PP}{\mathbf{P}}

\newcommand{\GG}{\mathbf{G}}

\newcommand{\cL}{\mathcal{L}}

\newcommand{\RP}{\mathbf{R}P}

\newcommand{\der}{\mathrm{der}}

\newcommand{\Nm}{\mathrm{Nm}}

\newcommand{\ul}[1]{\underline{#1}}
\newcommand{\ol}[1]{\overline{#1}}

\newcommand{\Perf}{\mathrm{Perf}}

\newcommand{\E}[1]{{\mathbf{E}_{{#1}}}}
\newcommand{\mmod}{/\!\!/}

\newcommand{\id}{\mathrm{id}}

\newcommand{\SL}{\mathrm{SL}}
\newcommand{\GL}{\mathrm{GL}}

\newcommand{\SO}{\mathrm{SO}}

\newcommand{\ku}{\mathrm{ku}}

\renewcommand{\H}{\mathrm{H}}

\newcommand{\Spin}{\mathrm{Spin}}

\newcommand{\Res}{\mathrm{Res}}
\newcommand{\Ind}{\mathrm{Ind}}

\newcommand{\fr}[1]{\mathfrak{#1}}
\newcommand{\g}{\mathfrak{g}}
\newcommand{\G}{\mathrm{G}}

\newcommand{\cP}{\mathcal{P}}

\newcommand{\Rep}{\mathrm{Rep}}

\newcommand{\cM}{\mathcal{M}}

\newcommand{\Shv}{\mathrm{Shv}}

\newcommand{\Higgs}{\mathrm{Higgs}}

\newcommand{\so}{\mathfrak{so}}

\newcommand{\PGL}{\mathrm{PGL}}

\renewcommand{\det}{\mathrm{det}}

\newcommand{\IC}{\mathrm{IC}}

\newcommand{\reg}{\mathrm{reg}}

\renewcommand{\S}{Section }

\makeatletter
\providecommand{\leftsquigarrow}{%
  \mathrel{\mathpalette\reflect@squig\relax}%
}
\newcommand{\reflect@squig}[2]{%
  \reflectbox{$\m@th#1\rightsquigarrow$}%
}
\makeatother

\newcommand{\bull}{\bullet}

\newcommand{\pw}[1]{[\![#1]\!]}
\newcommand{\ls}[1]{(\!(#1)\!)}

\newcommand{\act}{\circlearrowleft}

\newcommand{\rank}{\mathrm{rank}}

\newcommand{\sh}{\mathrm{sh}}

\renewcommand{\sl}{\mathfrak{sl}}

\renewcommand{\min}{\mathrm{min}}
\newcommand{\modc}{{\text{-}\mathrm{mod}}}

\newcommand{\Gr}{\mathrm{Gr}}

\newcommand{\rot}{\mathrm{rot}}

\newcommand{\ld}[1]{{\check{{#1}}}}

\usepackage{slashed}

\usepackage{relsize}
\usepackage[bbgreekl]{mathbbol}
\usepackage{amsfonts}
\DeclareSymbolFontAlphabet{\mathbb}{AMSb} 
\DeclareSymbolFontAlphabet{\mathbbl}{bbold}

\newcommand{\pdb}[1]{\langle{#1}\rangle}

\newcommand{\pgl}{\mathfrak{pgl}}

\newcommand{\std}{\mathrm{std}}

\newcommand{\Sat}{\mathrm{Sat}}

\newcommand{\PSO}{\mathrm{PSO}}
\newcommand{\DW}{\mathrm{DW}}
\newcommand{\ringcoeff}{\QQ}
\usepackage{dynkin-diagrams}

\title{Derived geometric Satake for $\PGL_2^{\times 3}/\PGL_2^\mathrm{diag}$}
\author{S. K. Devalapurkar}
\address{1 Oxford St, Cambridge, MA 02139}
\email{sdevalapurkar@math.harvard.edu, \today}
\thanks{Part of this work was done when the author was supported by NSF DGE-2140743}

\begin{document}

\maketitle

\begin{abstract}
    In this note, we study the local relative geometric Langlands conjecture of Ben-Zvi--Sakellaridis--Venkatesh for the spherical subgroup $\PGL_2^\mathrm{diag}$ of the triple product $\PGL_2^{\times 3}$ (and also for the spherical subgroup $\G_2$ of $\SO_8/\mu_2$), whose corresponding Langlands dual $\SL_2^{\times 3}$-variety can be identified with the symplectic vector space $(\AA^2)^{\otimes 3} \cong \AA^8$ of $2\times 2 \times 2$-cubes. Our analysis relies on a construction of Bhargava relating $2 \times 2 \times 2$-cubes to Gauss composition on quadratic forms, arising here as the moment map for the Hamiltonian $\SL_2^{\times 3}$-action on $(\AA^2)^{\otimes 3}$, and the Cayley hyperdeterminant as studied by Gelfand-Kapranov-Zelevinsky.
\end{abstract}

\tableofcontents

\section{Introduction}

The goal of this brief note is to study the geometrization of a story from the arithmetic context pioneered by B\"ocherer, Gross, Harris, Ichino, Kudla, Prasad, Schulze-Pillot, and Watson, among many others (see, e.g., \cite{harris-kudla, ichino-triple-product, prasad-thesis, watson-triple-product}): this tale is about the ``triple product period''.
In the language of geometric representation theory, our goal is to study derived geometric Satake (\cite[Conjecture 7.5.1]{bzsv}) for the spherical $\PGL_2^{\times 3}$-variety $\PGL_2^{\times 3}/\PGL_2^\mathrm{diag}$. This is a natural generalization of the ``group case'', i.e., derived geometric Satake for the spherical $\PGL_2^{\times 2}$-variety $\PGL_2^{\times 2}/\PGL_2^\mathrm{diag}$ (which amounts simply to the derived geometric Satake equivalence of \cite{bf-derived-satake} for $\PGL_2$). Similar methods allow us to analyze the spherical $\PSO_8$-variety $\PSO_8/\G_2$, too. In order to state our main result, some notation is necessary.

\begin{notation}
    Let $\std$ denote the standard representation of $\SL_2$, so that $\std^{\otimes 3}$ consists of cubes
    $$
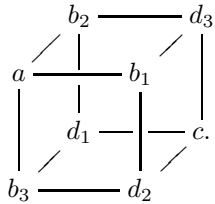
\xymatrix@=.75em{
    & b_2 \ar@{-}[rr] \ar@{-}'[d][dd] && d_3 \ar@{-}[dd] \\
    a \ar@{-}[ru] \ar@{-}[rr] \ar@{-}[dd] && b_1 \ar@{-}[ru] \ar@{-}[dd] & \\
    & d_1 \ar@{-}'[r][rr] && c. \\
    b_3 \ar@{-}[ru] \ar@{-}[rr] && d_2 \ar@{-}[ru] &
    }$$
    Fix an integer $n$.
    Equip $\std^{\otimes 3}$ with the grading where the entries of a cube have the following weights: $a$ lives in weight $-4n$, each $b_i$ lives in weight $-2n$, $c$ lives in weight $2n$, and each $d_i$ lives in weight $0$. Write $\std^{\otimes 3}(4n,\vec{2n},-2n,\vec{0})$ to denote the corresponding graded variety.
    
    Similarly, equip $\SL_2$ with the grading where the entries of a matrix $\begin{psmallmatrix}
        a & b\\
        c & d
    \end{psmallmatrix}$ have the following weights: $a$ and $d$ live in weight $0$, $b$ lives in weight $2n$, and $c$ lives in weight $-2n$. Write $\SL_2(-2n\rho)$ to denote this graded group. Then there is a natural graded action of $\SL_2(-2n\rho)^{\times 3}$ on $\std^{\otimes 3}(4n,\vec{2n},-2n,\vec{0})$.
\end{notation}
Recall that the process of \textit{shearing} (denoted $\sh^{1/2}$) discussed in \cite{raksit, rotinv}, as well as \cite[Section 2.1]{ku-rel-langlands}, converts gradings into homological shifts (more precisely, it sends a module in weight $n$ to the same module shifted homologically by $n$). This functor is symmetric monoidal when restricted to the subcategory of modules in \textit{even} weights, and therefore extends to an operation on evenly graded stacks. If $Y$ is a graded stack, let $\sh^{1/2} Y$ denote the corresponding derived stack obtained by shearing, and let $\Perf^\sh(Y)$ denote $\Perf(\sh^{1/2} Y)$. Almost all of our discussion below takes place in the setting of ordinary, and not derived, algebraic geometry, and so the reader unfamiliar with derived algebraic geometry should feel free to ignore this procedure of shearing.\footnote{The reader intent on ignoring shearing might feel some solace in noting that our main results should continue to hold with the $\infty$-category of constructible sheaves replaced by a mixed variant. However, this must (hopefully temporarily) be regarded as a reverie, since no ``mixed'' version of the derived geometric Satake equivalence seems to exist in the literature.}

As in \cite{ku-rel-langlands}, we will state all of our results with ``arithmetic shearing'' in the sense of \cite[Section 6.7]{bzsv}. The following result concerns \cite[Conjecture 7.5.1]{bzsv} for the Hamiltonian $\PGL_2^{\times 3}$-variety $T^\ast(\PGL_2^{\times 3}/\PGL_2^\mathrm{diag})$ and the Hamiltonian $\SL_2^{\times 3}$-variety $\std^{\otimes 3}$.
\begin{theorem}[Derived geometric Satake for $\PGL_2^{\times 3}/\PGL_2^\mathrm{diag}$]\label{thm: main}
    Suppose that the $\PGL_2^{\times 3}\pw{t}$-action on $\PGL_2^{\times 3}\ls{t}/\PGL_2^\mathrm{diag}\ls{t}$ is optimal in the sense of \cite[Hypothesis 3.5.2]{ku-rel-langlands}.
    There is an equivalence\footnote{The $\infty$-category on the left-hand side is as in \cite[Definition 3.6.1]{ku-rel-langlands}; see \cref{def: Shv-Sat LG/H} for a quick review.}
    $$\Shv_{\PGL_2^{\times 3}\pw{t}}^{c,\Sat}(\PGL_2^{\times 3}\ls{t}/\PGL_2^\mathrm{diag}\ls{t}); \ringcoeff) \simeq \Perf^\sh(\std^{\otimes 3}(4,\vec{2},-2,\vec{0})/\SL_2(-2\rho)^{\times 3}).$$
    Moreover, this equivalence is equivariant for the action of the spherical Hecke category for $\PGL_2^{\times 3}$; on the spectral side, this action is encoded by Bhargava's construction $\std^{\otimes 3} \to \sl_2^{\ast, \times 3}$ of three quadratic forms from a point of $\std^{\otimes 3}$ (see \cite{bhargava-composition-i, wood-thesis}).
\end{theorem}
\begin{remark}
    The hypothesis of optimality in \cref{thm: main} is equivalent to the hypothesis that the diagonal action of $\PGL_2\pw{t}$ on $\Gr_{\PGL_2} \times \Gr_{\PGL_2}$ is optimal, where $\Gr_{\PGL_2}$ is the affine Grassmannian of $\PGL_2$. In particular, \cite[Hypothesis 3.5.2(a) and (b)]{ku-rel-langlands} can be deduced from this observation. Checking purity of the $\IC$-sheaves seems a bit trickier, but should not be difficult.
\end{remark}
\begin{remark}
    Let $\PSO_{2n} := \SO_{2n}/\mu_2$. Then, the embedding $\PGL_2^\mathrm{diag} \subseteq \PGL_2^{\times 3}$ can be identified with the diagonal embedding $\SO_3 \subseteq \SO_3 \times \PSO_4$; and similarly, the action of $\SL_2^{\times 3}$ on $\std^{\otimes 3}$ can be identified with the action of $\Spin_4 \times \Sp_2$ on the tensor product of their respective defining representations. From this perspective, \cref{thm: main} could be viewed as a special case of the geometrized analogue of the Gan-Gross-Prasad period (or at least a period isogenous to it). See \cref{rmk: ggp generalization} (as well as \cite[Example 3.6.24]{ku-rel-langlands}) for a discussion of the Gan-Gross-Prasad period along the lines of this article.
\end{remark}
See \cref{rmk: so8 minimal nilpotent} for an analogue of \cref{thm: main} for the non-spherical variety $\PGL_2^{\times 4}/\PGL_2^\mathrm{diag}$.

An argument similar to that of \cref{thm: main} shows a variant for $\PSO_8$. (The arithmetic version of the result stated below was studied, for instance, in \cite{jiang-g2-periods, jiang-triple-product}.) Namely, there is an embedding $\G_2 \subseteq \PSO_8$ given by triality, which exhibits $\G_2$ as a spherical subgroup of $\PSO_8$. To see that this situation is analogous to that of \cref{thm: main}, note that the Dynkin diagram $\bull$ of $A_1$ is obtained from the Dynkin diagram $\bull \bull \bull$ of $A_1^{\times 3}$ by folding with respect to the obvious action of the symmetric group $\Sigma_3$. In the same way, the Dynkin diagram \dynkin{G}{2} of $\G_2$ is obtained from the Dynkin diagram \dynkin{D}{4} of $D_4$ by folding with respect to the action of $\Sigma_3$ permuting the three vertices around the branching vertex. The following result concerns \cite[Conjecture 7.5.1]{bzsv} for the Hamiltonian $\PSO_8$-variety $T^\ast(\PSO_8/\G_2)$ and the Hamiltonian $\Spin_8$-variety $\Ind_{\SL_2^{\times 3}}^{\Spin_8}(\AA^1\oplus \std^{\otimes 3})$.
\begin{theorem}[Derived geometric Satake for $\PSO_8/\G_2$]\label{thm: G2 main}
    Suppose that the $\PSO_8\pw{t}$-action on $\PSO_8\ls{t}/\G_2\ls{t}$ is optimal in the sense of \cite[Hypothesis 3.5.2]{ku-rel-langlands}. 
    Then there is an equivalence
    $$\Shv_{\PSO_8\pw{t}}^{c,\Sat}(\PSO_8\ls{t}/\G_2\ls{t}; \ringcoeff) \simeq \Perf^\sh(\std^{\otimes 3}(12,\vec{6},-6,\vec{0})/\SL_2(-6\rho)^{\times 3} \times \AA^1(4)).$$
\end{theorem}
This implies, for instance, that the spherical subgroups $\PGL_2^{\times 2} \subseteq \PGL_2^{\times 4}$ (given by $(g,h)\mapsto (g,g,g,h)$) and $\G_2 \subseteq \PSO_8$ have the same dual quotient stacks (namely, $\std^{\otimes 3}/\SL_2^{\times 3} \times \AA^1$) up to grading. Therefore, they fit into the paradigm of \cite[Remark 4.1.5]{ku-rel-langlands}.

The proofs of \cref{thm: main} and \cref{thm: G2 main} reduce to showing that the conditions of \cite[Theorem 3.6.4]{ku-rel-langlands} are met.
This ultimately relies on studying Bhargava's construction from \cite{bhargava-composition-i} relating $2\times 2 \times 2$-matrices to quadratic forms, and the work \cite{gelfand-hyperdet} of Gelfand-Kapranov-Zelevinsky describing the relationship to Cayley's hyperdeterminant. The work presented in this article indicates that there is much more to explore regarding the relationship between (relative) geometric Langlands (following \cite{bzsv}), prehomogeneous vector spaces (following \cite{sato-kimura-prehomogeneous}), and the progress in arithmetic invariant theory (using the terminology of \cite{bhargava-gross-arithmetic-invariant}) over the past 20 years spurred by Bhargava's thesis. To illustrate this, here are two concrete questions in the interface of these subjects:
\begin{itemize}
    \item Can one use Bhargava's work \cite{bhargava-composition-i} on the $\SL_6(\Z)$-action on $\wedge^3 \Z^6$ to approach the geometrized version of \cite{ginzburg-rallis-exterior-cube} (i.e., \cite[Line 16 of Table 1.5.1]{bzsv})? (Here, the information flows from arithmetic invariant theory to relative geometric Langlands.)
    \item Following the work of Moore and Tachikawa in \cite{moore-tachikawa}, let $\ol{\co} := \ol{\co_\min}(\fr{so}_8)$ denote the minimal nilpotent coadjoint orbit of $\fr{so}_8$, so that it is $10$-dimensional. It has a canonical action of $\SO_8$, and hence a canonical action of $\SL_2^{\times 4} \subseteq \SO_8$. In fact, more is true: there is an action of the symmetric group $\Sigma_4$ on $\fr{so}_8$ extending the $\Sigma_3\subseteq \Sigma_4$-action by triality, whose fixed subgroup is $\sl_3 \subseteq \g_2$. (See \cite[Section 4.1]{S4-action-SO8} for a description of this action.) This defines an action of $\Sigma_4$ on $\ol{\co}$, and in fact the action of $\SL_2^{\times 4}$ on $\ol{\co}$ is $\Sigma_4$-equivariant for the permutation action on $\SL_2^{\times 4}$. 
    
    The action $\SL_2^{\times 4} \act \ol{\co}$ should be viewed as the analogue of $\SL_2^{\times 3} \act (\AA^2)^{\otimes 3}$; see also \cref{rmk: so8 minimal nilpotent}.\footnote{Here is another (related) reason. As described in \cite[Remark 3.6.25]{ku-rel-langlands}, both of these examples can be placed in the following context. The groups $(\SL_2,\SL_2^{\times 2})$ form a reductive dual pair in $\Sp_8$, and $(\AA^2)^{\otimes 3}$ is a double cover of the minimal nilpotent orbit closure of $\Sp_8$. Similarly, the groups $(\SL_2,\SL_2^{\times 3})$ form a reductive dual pair in $\SO_8$, and $\ol{\co}$ is the minimal nilpotent orbit closure of $\SO_8$. One should therefore view these examples as falling under the purview of Howe duality.} There is a map $\ol{\co} \to \AA^1$, which we will call the ``discriminant'', which defines an isomorphism
    $$\ol{\co}\mmod \SL_2^{\times 4} \xar{\sim} \AA^1,$$
    at least over $\cc$. The above isomorphism should also hold true over $\Z$.
    
    Following \cite{bhargava-composition-i}, let us introduce some terminology. An ordered tuple $(I_1, I_2, I_3, I_4)$ of oriented ideals of a quadratic ring $S$ such that $I_1 I_2 I_3 I_4 \subseteq S$ and $N(I_1) N(I_2) N(I_3) N(I_4) = 1$ will be called \textit{balanced}, and say that two such balanced tuples $(I_1, I_2, I_3, I_4)$ and $(I_1', I_2', I_3', I_4')$ are equivalent if there are elements $x_1, \cdots, x_4\in S \otimes \QQ$ such that $I_j' = x_j I_j$.
    Say that an $\SL_2^{\times 4}(\Z)$-orbit on $\ol{\co}(\Z)$ is \textit{nondegenerate} if its discriminant is nonzero.
    Then, one should have the following analogue of \cite[Theorem 11]{bhargava-composition-i}: 
    \begin{conjecture}
        There is a canonical bijection between the set of nondegenerate $\SL_2^{\times 4}(\Z)$-orbits on $\ol{\co}(\Z)$ and the set of isomorphism classes of pairs $(S,(I_1, I_2, I_3, I_4))$ of a nondegenerate oriented quadratic ring $S$ and an equivalence class $(I_1, I_2, I_3, I_4)$ of balanced oriented ideals of $S$.
    \end{conjecture}
    If this were true, one could symmetrize, i.e., take $I_1 = I_2 = I_3 = I_4$, so that $I := I_1$ will be a $4$-torsion element in the class group of $S$. It would be very interesting if one could use the explicit description of the invariant quotient $\ol{\co}\mmod \Sigma_4$ in \cite[Theorem 4.6 and Theorem 4.12]{S4-action-SO8} to describe the asymptotics of $4$-torsion in class groups of quadratic rings over $\Z$. (Here, the information flows from relative geometric Langlands to arithmetic invariant theory.)
\end{itemize}
In \cref{rmk: SOn more general} and \cref{rmk: pairs of binary cubics}, I have included some further natural observations/questions whose role in the picture of geometric Langlands is unclear to me, in the hopes that it will stimulate further investigation into these questions.

\begin{remark}
    The arguments of this article should continue to hold if one considers sheaves with coefficients in $\Z[\frac{1}{2}]$; we have not checked this explicitly, but it seems likely to be true. In fact, we expect that the results of this article should continue to hold for sheaves with coefficients in $\Z$ itself. This, however, is a rather more subtle question: the prime $2$ is an interesting one (see \cref{rmk: prime 2 cayley}).

    More generally, following the philosophy of \cite{ku-rel-langlands}, it should also be possible to use a variant of the methods of this article to prove analogues of \cref{thm: main} and \cref{thm: G2 main} for sheaves with coefficients in connective complex K-theory $\ku$. We have not attempted to do this, but we expect the corresponding $1$-parameter deformation of $\std^{\otimes 3}$ over $\pi_\ast(\ku) \cong \Z[\beta]$ to be a rather interesting $\ku$-Hamiltonian $\SL_2^{\times 3}$-variety.
\end{remark}
\begin{remark}
    The equivalence of \cref{thm: main} can heuristically be viewed as geometric Langlands for $\PGL_2$ on the ``doubled raviolo'', obtained by gluing three formal disks along their common punctured disk.
    I expect \cref{thm: main} to be related to the work of \cite{moore-tachikawa}, and hope to address this relationship and the context of \cref{thm: main} in physics in joint work with Ben-Zvi and Gunningham.

    \cref{thm: main} is also related to the monoidality of the derived Satake equivalence.
    The $\infty$-category $\Shv_{\PGL_2^{\times 3}\pw{t}}^{c,\Sat}(\PGL_2^{\times 3}\ls{t}/\PGL_2^\mathrm{diag}\ls{t}); \ringcoeff)$ can be identified with $\Shv_{\PGL_2\pw{t}}^{c,\Sat}(\Gr_{\PGL_2} \times \Gr_{\PGL_2}; \ringcoeff)$, where $\Gr_{\PGL_2}$ is the affine Grassmannian of $\PGL_2$. In particular, it can be used to describe the \textit{ordinary} tensor product on $\Shv_{\PGL_2\pw{t}}^{c,\Sat}(\Gr_{\PGL_2}; \ringcoeff)$. \cref{thm: main} says that under the derived geometric Satake equivalence of \cite{bf-derived-satake} identifying this $\infty$-category with $\Perf^\sh(\Sym^2(\std)/\SL_2)$, the ordinary tensor product on $\Shv_{\PGL_2\pw{t}}^{c,\Sat}(\Gr_{\PGL_2}; \ringcoeff)$ can be identified with Gauss composition of binary quadratic forms (in the way described by Bhargava), i.e., via the correspondence \cref{eq: std tensor 3 span}.\footnote{This is similar to how the \textit{convolution} tensor product on $\Shv_{\PGL_2\pw{t}}^{c,\Sat}(\Gr_{\PGL_2}; \ringcoeff)$ can be identified with the usual tensor product on $\Perf^\sh(\Sym^2(\std)/\SL_2)$.}
    (In fact, I learned after writing this article that the main calculation of this article also appears in \cite[Section 5(iii)]{bfn-ring-objects}. The approach taken here is somewhat different, in that the calculations are more explicit and fit into the story of \cite{ku-rel-langlands}; I hope that the reader might find this alternative perspective useful.)

    In the final section of this article, we suggest some variants of \cref{thm: main} with $\PGL_2^{\times 3}$ replaced by variants. Namely, we expect:
    \begin{itemize}
        \item Let $G = \Res_{\cc\pw{t^{1/3}}/\cc\pw{t}} \PGL_2$, where $\PGL_2$ is viewed as a constant group scheme over $\cc\pw{t^{1/3}}$. Then there should be a fully faithful functor from the $\infty$-category of perfect complexes on a shearing of the quotient stack $\Sym^3(\std)/\SL_2$ to $\Shv_{G\pw{t}}^{c,\Sat}(G\ls{t}/\PGL_2\ls{t}; \ringcoeff)$. The grading with respect to which the shearing is taken is described in \cref{conj: langlands binary cubics}.
        \item Let $G = \Res_{(\cc\pw{t^{1/2}} \times \cc\pw{t})/\cc\pw{t}} \PGL_2$. Then there should be a fully faithful functor from the $\infty$-category of perfect complexes on a shearing of the quotient stack $(\std \otimes \sl_2^\ast)/\SL_2^{\times 2}$ to  $\Shv_{G\pw{t}}^{c,\Sat}(G\ls{t}/\PGL_2\ls{t}; \ringcoeff)$. The grading with respect to which the shearing is taken is described in \cref{conj: langlands pairs of binary quadratic}.
    \end{itemize}
\end{remark}
\begin{remark}\label{rmk: intro-quantum}
    The quotient stack $\std^{\otimes 3}/\SL_2^{\times 3}$ is also studied (in different language, of course) in quantum information theory; see \cref{rmk: quantum information} below.
\end{remark}
\cref{thm: main} and \cref{thm: G2 main} are predicted by (the Betti version of) the local geometric conjecture of Ben-Zvi--Sakellaridis--Venkatesh; see \cite[Conjecture 7.5.1]{bzsv}. My homotopy-theoretic interpretation of their conjecture is as follows. Suppose $G$ is a reductive group over $\cc$ and $G/H$ is an affine homogeneous spherical $G$-variety (meaning that it admits an open $B$-orbit for its natural left $B\subseteq G$-action). Then, there should be a dual graded $\ld{G}$-variety $\ld{M}$ equipped with a moment map $\mu: \ld{M} \to \ld{\g}^\ast$, and an equivalence of the form
$$\Shv_{G\pw{t}}^{c,\Sat}(G\ls{t}/H\ls{t}; \cc) \simeq \Perf^\sh(\ld{M}/\ld{G}),$$
where $\Perf^\sh$ denotes the $\infty$-category of perfect complexes on the shearing of $\ld{M}/\ld{G}$ with respect to its given grading. In fact, \cite[Section 4]{bzsv} gives an explicit construction of this predicted dual $\ld{M}$, and in the examples $(G, H) = (\PGL_2^{\times 3}, \PGL_2^\mathrm{diag})$ and $(\PSO_8, \G_2)$, one can compute that the stacky quotient $\ld{M}/\ld{G}$ is isomorphic to the right-hand sides of \cref{thm: main} and \cref{thm: G2 main} respectively.\footnote{In the first case, this computation is straightforward given the prescription of \cite[Section 4]{bzsv}; see \cite[Example 7.2.4]{sakellaridis-spherical-functions} for a reference. The computation in the second case goes as follows. As in \cite[Remark 7.1.1]{bzsv}, the quotient stack $\ld{M}/\ld{G}$ can be identified with the quotient $\ld{V}_X/\ld{G}_X$, where $\ld{G}_X$ is the Gaitsgory-Nadler/Sakellaridis-Venkatesh/Knop-Schalke dual group of $X$ and $\ld{V}_X$ is constructed in \cite[Section 4.5]{bzsv}. In the case $X = \PSO_8/\G_2$, a calculation shows that $\ld{G}_X$ is the Levi subgroup of the maximal parabolic subgroup of $\PSO_8$ corresponding to the central vertex of the $D_4$ Dynkin diagram; so $\ld{G}_X \cong \SL_2^{\times 3}$. Using the prescription of \cite[Section 4.5]{bzsv}, one can check that $\ld{V}_X \cong \std^{\otimes 3} \oplus \AA^1$, where $\ld{G}_X$ acts only on the first factor. See, e.g., \cite[Line 9 of Table in Appendix A]{sakellaridis-spherical-functions}.}

Lest \cref{thm: main} seem like an oddly specific example to focus on, we note that it is essentially the \textit{only} ``new'' example of a spherical pair $(G, H)$ of the form $(H^{\times j}, H^\mathrm{diag})$, as shown by the following lemma.
\begin{lemma}\label{lem: when is diagonal spherical}
    Suppose $H$ is a simple linear algebraic group over $\cc$. Then the subgroup $H^\mathrm{diag} \subseteq H^{\times j}$ is spherical if and only if:
    \begin{enumerate}
        \item $j = 2$, and $H$ arbitrary;
        \item $j = 3$ and $H$ is of type $A_1$.
    \end{enumerate}
\end{lemma}
\begin{proof}
    If the subgroup $H^\mathrm{diag} \subseteq H^{\times j}$ is spherical, there is an open $H^\mathrm{diag}$-orbit on the flag variety of $H^{\times j}$. This implies that the dimension of $H$ must be at least $j|\Phi^+|$, where $\Phi^+$ is the set of positive roots; equivalently, one needs $\mathrm{rank}(H) \geq (j-2)|\Phi^+|$.
    Of course, this is always satisfied if $j = 2$ (this is the group case corresponding to the symmetric subgroup $H^\mathrm{diag} \subseteq H \times H$).
    Using the classification of simple linear algebraic groups over $\cc$, it is easy to see that the only other case when the above inequality can hold is when $j = 3$ and $H$ is of type $A_1$; one can then check by hand that the diagonal subgroup in this case is indeed spherical.
\end{proof}
In the first case of \cref{lem: when is diagonal spherical}, \cite[Conjecture 7.5.1]{bzsv} is precisely the derived geometric Satake equivalence of \cite{bf-derived-satake}.
Therefore, the only other case of \cref{lem: when is diagonal spherical} is when $H$ is simple of type $A_1$, and \cref{thm: main} precisely addresses \cite[Conjecture 7.5.1]{bzsv} for the adjoint form $\PGL_2$ of $H$.

\subsection*{Acknowledgements}

I am very grateful to Yiannis Sakellaridis for his support and help in answering my numerous questions about \cite{bzsv}; to Aaron Landesman, Ashvin Swaminathan, and Akshay Venkatesh for interesting conversations; to Alison Miller for directing me (via MathOverflow) to the inspirational work of Bhargava; and to Charles Fu and Jesper Grodal for helpful suggestions.

\section{Some properties of $\std^{\otimes 3}$}

In this section, we establish some basic properties of $\std^{\otimes 3}$ as a $\SL_2^{\times 3}$-variety; our base field will always be $\ringcoeff$, and we will write $\ld{G} = \SL_2^{\times 3}$. Some of this material appears in \cite{bhargava-composition-i, wood-thesis}. In particular, \cref{def: cube and quadratic forms} is due to Bhargava.
\begin{observe}\label{obs: sl2 and binary quadratic}
    An element $A = \begin{psmallmatrix}
        a & b\\
        c & -a
    \end{psmallmatrix} \in \sl_2^\ast \cong \sl_2$ can be identified with a binary quadratic form $q_A(x,y) = cx^2 + 2axy - by^2$. The resulting identification between $\Sym^2(\std) = (\std \otimes \std)_{\Sigma_2}$ and $\sl_2$ is $\SL_2$-equivariant.
    Note, moreover, that the discriminant of $q_A(x,y)$ is $-4\det(A)$.
\end{observe}
\begin{warning}
    Note that under \cref{obs: sl2 and binary quadratic}, the element of $\sl_2^\ast$ associated to a binary quadratic form $bx^2 + axy + cy^2$ is \textit{not} the symmetric matrix associated to the quadratic form! Indeed, the associated symmetric matrix is $\begin{psmallmatrix}
        b & a/2 \\
        a/2 & c
    \end{psmallmatrix}$, while the associated element of $\sl_2^\ast$ is $\begin{psmallmatrix}
        a/2 & c\\
        b & -a/2
    \end{psmallmatrix}$.
    
    Note, also, that we are relying quite heavily on the assumption that $2$ is invertible in $\ringcoeff$. Over $\Z$, one can in fact identify the space $(\AA^2 \otimes \AA^2)_{\Sigma_2}$ of binary quadratic forms with the \textit{coadjoint} representation $\sl_2^\ast \cong \pgl_2$ of $\SL_2$. Working over $\Z$ and keeping track of the difference between $\sl_2$ and $\pgl_2$ has the effect of eliminating extraneous factors of $2$ in our discussion below; but working over $\Z$ also introduces new complications (see \cref{rmk: prime 2 cayley}) which we do not wish to address in the present article. 
\end{warning}
\begin{construction}\label{def: cube and quadratic forms}
    The affine space $\AA^8 = \std^{\otimes 3}$ can be regarded as parametrizing cubes
    $$\xymatrix@=.75em{
    & b_2 \ar@{-}[rr] \ar@{-}'[d][dd] && d_3 \ar@{-}[dd] \\
    a \ar@{-}[ru] \ar@{-}[rr] \ar@{-}[dd] && b_1 \ar@{-}[ru] \ar@{-}[dd] & \\
    & d_1 \ar@{-}'[r][rr] && c, \\
    b_3 \ar@{-}[ru] \ar@{-}[rr] && d_2 \ar@{-}[ru] &
    }$$
    which we will represent by a tuple $(a, \vec{b}, c, \vec{d})$; we will often use the symbol $\cC$ to denote such a cube. If $\{e_1, e_2\}$ are a basis for $\std$, the above cube corresponds to the element of $\std^{\otimes 3}$ given by
    \begin{align*}
        a e_1 \otimes e_1 \otimes e_1 & + b_1 e_2 \otimes e_1 \otimes e_1 + b_2 e_1 \otimes e_2 \otimes e_1 + b_3 e_1 \otimes e_1 \otimes e_2 \\
        + d_1 e_1 \otimes e_2 \otimes e_2 & + d_2 e_2 \otimes e_1 \otimes e_2 + d_3 e_2 \otimes e_2 \otimes e_1 + c e_2 \otimes e_2 \otimes e_2.
    \end{align*}
    Associated to a cube $\cC$ are three pairs of matrices, given by slicing along the top, leftmost, or front faces:
    \begin{align*}
        M_1 = \begin{psmallmatrix}
            a & b_2\\
            b_3 & d_1
        \end{psmallmatrix}, \ & N_1 = \begin{psmallmatrix}
            b_1 & d_3\\
            d_2 & c
        \end{psmallmatrix}, \\
        M_2 = \begin{psmallmatrix}
            a & b_1\\
            b_3 & d_2
        \end{psmallmatrix}, \ & N_2 = \begin{psmallmatrix}
            b_2 & d_3\\
            d_1 & c
        \end{psmallmatrix}, \\
        M_3 = \begin{psmallmatrix}
            a & b_1\\
            b_2 & d_3
        \end{psmallmatrix}, \ & N_3 = \begin{psmallmatrix}
            b_3 & d_2\\
            d_1 & c
        \end{psmallmatrix};
    \end{align*}
    each of these defines a binary quadratic form
    $$q_i(x,y) = -\det(M_i x + N_i y).$$
    Explicitly,
    \begin{align*}
        q_1(x,y) & = \det(M_1) x^2 + (ac + b_1 d_1 - b_2 d_2 - b_3 d_3) xy + \det(N_1) y^2, \\
        q_2(x,y) & = \det(M_2) x^2 + (ac - b_1 d_1 + b_2 d_2 - b_3 d_3) xy + \det(N_2) y^2,\\
        q_3(x,y) & = \det(M_3) x^2 + (ac - b_1 d_1 - b_2 d_2 + b_3 d_3) xy + \det(N_3) y^2.
    \end{align*}
    Viewing $\sl_2^\ast$ as the space of binary quadratic forms as in \cref{obs: sl2 and binary quadratic}, these three quadratic forms define a map
    $$\mu: \std^{\otimes 3} \to \sl_2^{\ast, \times 3}.$$
    An easy check shows that this map is $\ld{G}$-equivariant.
\end{construction}

\begin{lemma}[{Cayley, \cite{cayley-original}}]\label{lem: cayley}
    The composite
    $$\std^{\otimes 3} \xar{\mu} \sl_2^{\ast, \times 3} \to \sl_2^{\ast, \times 3}\mmod \ld{G}$$
    factors through the diagonal inclusion $\sl_2^\ast\mmod \SL_2 \to \sl_2^{\ast, \times 3}\mmod \ld{G}$. In fact, the induced map $\std^{\otimes 3} \to \sl_2^\ast\mmod \SL_2$ defines an isomorphism
    $$\std^{\otimes 3}\mmod \ld{G} \xar{\sim} \sl_2^\ast\mmod \SL_2 \cong \AA^1\mmod (\Z/2).$$
\end{lemma}
\begin{proof}
    The map $\sl_2^{\ast, \times 3} \to \sl_2^{\ast, \times 3}\mmod \ld{G}$ sends a triple of matrices to their determinants, or equivalently a triple of quadratic forms to their discriminants. Therefore, we need to check that the three quadratic forms of \cref{def: cube and quadratic forms} have the same discriminant. This is easy: one finds that their common discriminant is
    \begin{align}
        \det(q_i) & = a^2 c^2 + b_1^2 d_1^2 + b_2^2 d_2^2 + b_3^2 d_3^2 - 2 (a b_1 c d_1 + a b_2 c d_2 + a b_3 c d_3 \nonumber \\
        & + b_1 b_2 d_1 d_2 + b_1 b_3 d_1 d_3 + b_2 b_3 d_2 d_3) + 4 (a d_1 d_2 d_3 + b_1 b_2 b_3 c). \label{eq: common disc}
    \end{align}
    It remains to show that the map $\std^{\otimes 3}\mmod \ld{G} \to \AA^1$ defined by this polynomial is an isomorphism. This is stated/proved in \cite[Proposition 1.7 in Chapter 14]{gelfand-hyperdet}, and is due to Cayley (see \cite[Page 89]{cayley-original} for the original source!).
\end{proof}
\begin{notation}
    Write $\det$ to denote the map $\std^{\otimes 3}\to \sl_2^\ast\mmod \SL_2$ from \cref{lem: cayley}, so that if $\cC$ is a cube, $\det(\cC)$ is the quantity of \cref{eq: common disc}. 
\end{notation}

\begin{remark}\label{rmk: sympl on C2 cubed}
    The standard $\SL_2$-equivariant symplectic structure on $\std$ defines an $\SL_2^{\times 3}$-equivariant symplectic structure on $\std^{\otimes 3}$. This action is Hamiltonian, and one can verify by explicit calculation that the map $\mu: \std^{\otimes 3} \to \sl_2^{\ast, \times 3}$ from \cref{def: cube and quadratic forms} is in fact simply the moment map for this $\SL_2^{\times 3}$-action. In other words, if $\omega_1, \omega_2, \omega_3$ denote the standard symplectic forms on the three copies of $\std$, and $\omega = \omega_1 \otimes \omega_2 \otimes \omega_3$ is the resulting symplectic form on $\std^{\otimes 3}$, one can identify
    $$\mu: \std^{\otimes 3} \to \sl_2^{\ast, \times 3}, \ v\mapsto \left[\xi \mapsto \tfrac{1}{2}\omega(v, \xi v)\right].$$
    This gives a more ``invariant'' way to think about Bhargava's three quadratic forms.
    
    Along these lines, let us remark that \cite[Theorem 1]{bhargava-composition-i} implies that the span
    \begin{equation}\label{eq: std tensor 3 span}
        \xymatrix{
        & \std^{\otimes 3} \ar[dl]_-{\mu_1 \times \mu_2} \ar[dr]^-{\mu_3} & \\
        \sl_2^\ast \times_{\sl_2^\ast\mmod \SL_2} \sl_2^\ast & & \sl_2^\ast
        }
    \end{equation}
    given by the moment maps \textit{encodes} Gauss composition on quadratic forms, in the sense that given two ($\SL_2$-orbits of) quadratic forms $q_1$ and $q_2$ with the same discriminant, the ($\SL_2$-orbit of) the Gauss composition $-(q_1 + q_2)$ is given by $\mu_3((\mu_1 \times \mu_2)^{-1}(q_1, q_2))$. The above span induces a span of stacks
    $$\xymatrix{
    & \std^{\otimes 3}/\SL_2^{\times 3} \ar[dl]_-{\mu_1 \times \mu_2} \ar[dr]^-{\mu_3} & \\
    (\sl_2^\ast \times_{\sl_2^\ast\mmod \SL_2} \sl_2^\ast)/(\SL_2 \times \SL_2) & & \sl_2^\ast/\SL_2,
    }$$
    which, using the isomorphism $\Spin_4 \cong \SL_2 \times \SL_2$, can be identified with the Lagrangian correspondence \cite[Corollary 3.6.20]{ku-rel-langlands} for $H = \SO_3$ sitting inside $G = \PSO_4$.
\end{remark}

\begin{remark}\label{rmk: cayley as actual det}
    An alternative way of constructing $\det(\cC)$ is as follows. Write $\cC = e_1 \otimes v_1 + e_2 \otimes v_2$ with $v_1, v_2\in \std^{\otimes 2} \cong \AA^4$, and consider the symmetric bilinear form on $\std^{\otimes 2}$ given by
    $$\pdb{e_1 \otimes e_1, e_2 \otimes e_2} = -\pdb{e_1 \otimes e_2, e_2 \otimes e_1} = 1,$$
    and all other pairings zero. This is the symmetric form on $\std^{\otimes 2}$ induced from the standard symplectic form on $\std$. Then, one can identify
    $$\det(\cC) = \det \begin{psmallmatrix}
        \pdb{v_1, v_1} & \pdb{v_1, v_2} \\
        \pdb{v_2, v_1} & \pdb{v_2, v_2}
    \end{psmallmatrix}.$$
\end{remark}

\begin{remark}\label{rmk: prime 2 cayley}
    \cref{lem: cayley} is not quite true over $\FF_2$ (and hence not over $\Z$).
    One can already see the subtlety that arises over $\FF_2$ from the formula \cref{eq: common disc}: namely, the Cayley hyperdeterminant (appropriately normalized) admits a square root over $\FF_2$. Explicitly, if $\cC = (a, \vec{b}, c, \vec{d})$ is a cube and $\det(\cC)$ is defined by the formula \cref{eq: common disc}, one has
    $$\frac{\det(\cC)}{2} \equiv \frac{(ac + b_1 d_1 + b_2 d_2 + b_3 d_3)^2}{2} \pmod{2}.$$
    In fact, over an $\FF_2$-algebra, there is an analogue of \cref{lem: cayley} which states that the composite
    $$\std^{\otimes 3} \to \sl_2^{\ast, \times 3} \cong \fr{pgl}_2^{\times 3} \xar{\Tr} (\AA^1)^{\times 3}$$
    factors through the diagonal $\AA^1 \subseteq (\AA^1)^{\times 3}$; the resulting map $\std^{\otimes 3} \to \AA^1$ is given by the $\ld{G}$-invariant quadratic function
    $$\cC \mapsto \Tr(\cC) := ac + b_1 d_1 + b_2 d_2 + b_3 d_3.$$
    I expect that $\Tr$ defines an isomorphism $\std^{\otimes 3} \mmod \ld{G} \to \AA^1$ over $\FF_2$, i.e., $\H^0(\ld{G}; \Sym(\std^{\otimes 3, \ast})) \cong \FF_2[\Tr]$.
    In particular, this means that the Cayley hyperdeterminant does \textit{not} define an isomorphism $\det: \std^{\otimes 3}\mmod \ld{G} \to \AA^1$ over $\FF_2$.
    
    However, I expect more to be true (this is based on discussions with Akshay Venkatesh). Namely, the moment map should induce an isomorphism of derived schemes
    $$\xymatrix{
    \std^{\otimes 3}\mmod_\der \ld{G} \ar[dr]^-\mu \ar[d]_-{\sim} \\
    \sl_2^\ast\mmod_\der \SL_2 \ar[r]_-{\mathrm{diag}} & (\sl_2^\ast\mmod_\der \SL_2)^{\times 3}
    }$$
    even over $\Z$, where the symbol $\mmod_\der$ denotes the \textit{derived} invariant-theoretic quotient (i.e., $V\mmod_\der H = \spec R\Gamma(BH; \Sym(V^\ast))$). We also expect that 
    $$\H^\ast(\SL_{2,\Z}; \sh^{1/2} \Sym(\sl_{2,\Z}(2))) \cong \H^\ast_{\SO_3}(\ast; \Z).$$
    There is an isomorphism $\H^\ast_{\SO_3}(\ast; \Z) \cong \Z[p_1, e]/2e$, where the Euler class $e$ lives in cohomological degree $3$. For example, the class $p_1$ should correspond to the determinant $\sl_2^\ast \to \AA^1$, and the fact that it has a square root modulo $2$ corresponds to the fact that the determinant map admits a square root over $\FF_2$, given by the trace $\sl_2^\ast \cong \pgl_2 \to \AA^1$. Similarly, the class $e$ should correspond to the nontrivial extension of $\sl_2^\ast \cong \Sym^2(\std)$ given by $\fr{gl}_2 \cong \std \otimes \std$.
\end{remark}

We will now define an analogue of the Kostant slice, as it will be needed to apply \cite[Theorem 3.6.4]{ku-rel-langlands} (see \cite[Strategy 1.2.1(b)]{ku-rel-langlands}). For the purposes of our discussion, one should view this Kostant section as an analogue of the construction of the companion matrix associated to a characteristic polynomial.
\begin{construction}\label{cstr: kostant}
    If $n$ is an integer, let $\vec{n}$ denote the triple $(n,n,n)$. 
    Let 
    $$\kappa: \sl_2^\ast\mmod \SL_2\cong \AA^1\mmod (\Z/2) \cong \AA^1 \to \std^{\otimes 3}$$
    denote the map sending $a^2\mapsto (a^2,\vec{0},0,\vec{1})$.
    This corresponds to the cube
    $$\xymatrix@=.75em{
    & 0 \ar@{-}[rr] \ar@{-}'[d][dd] && 1 \ar@{-}[dd] \\
    a^2 \ar@{-}[ru] \ar@{-}[rr] \ar@{-}[dd] && 0 \ar@{-}[ru] \ar@{-}[dd] & \\
    & 1 \ar@{-}'[r][rr] && 0. \\
    0 \ar@{-}[ru] \ar@{-}[rr] && 1 \ar@{-}[ru] &
    }$$
    In this case, $\det(\kappa(a^2)) = 4a^2$, so that $\kappa$ defines a section of $\det$ (at least up to the unit $4\in k^\times$). The associated quadratic forms are all equal, and are given by
    $$q_1(x,y) = q_2(x,y) = q_3(x, y) = a^2 x^2 - y^2,$$
    which corresponds to the traceless matrix $\begin{psmallmatrix}
        0 & 1\\
        a^2 & 0
    \end{psmallmatrix} \in \sl_2^\ast$. (Note that this is exactly the companion matrix associated to the characteristic polynomial $y^2 - a^2$.)
\end{construction}

\begin{remark}
    Fix an integer $n$. Then the $\ld{G}$-variety $\std^{\otimes 3}$ admits a natural grading, where the entries of a cube $(a,\vec{b},c,\vec{d})$ have the following weights: $a$ lives in weight $-4n$, $b$ lives in weight $-2n$, $c$ lives in weight $2n$, and $d$ lives in weight $0$. Write $\std^{\otimes 3}(4n,\vec{2n},-2n,\vec{0})$ to denote the associated graded variety.
    Equip $\sl_2^\ast$ with the grading where the entries of a matrix $\begin{psmallmatrix}
        a & b\\
        c & -a
    \end{psmallmatrix}$ have the following weights: $a$ lives in weight $-2n$, $b$ lives in weight $0$, and $c$ lives in weight $-4n$. Similarly, equip $\SL_2$ with the grading coming from $2n\rho$, so that the entries of a matrix $\begin{psmallmatrix}
        a & b\\
        c & d
    \end{psmallmatrix}$ have the following weights: $a$ and $d$ live in weight $0$, $b$ lives in weight $2n$, and $c$ lives in weight $-2n$.
    With these gradings, the $\SL_2^{\times 3}$-equivariant map $\mu: \std^{\otimes 3} \to \sl_2^{\ast, \times 3}$ is a \textit{graded} map, and $\kappa$ defines a graded map $\sl_2^\ast(2n)\mmod \SL_2 \cong \AA^1(4n) \to \std^{\otimes 3}(4n,\vec{2n},-2n,\vec{0})$.
    The cases $n=1$ and $n=3$ will be relevant below (corresponding to \cref{thm: main} and \cref{thm: G2 main}, respectively).
\end{remark}

\begin{remark}\label{rmk: quantum information}
    As mentioned in \cref{rmk: intro-quantum}, the quotient stack $\std^{\otimes 3}/\SL_2^{\times 3}$ is studied in quantum information theory. For instance, in \cite{dur-vidal-cirac}, D\"ur-Vidal-Cirac study the orbit structure of $\SL_2^{\times 3}$ acting on $\std^{\otimes 3}$ (in particular, they recover \cref{fig: orbits C2 cubed} independently of \cite{gelfand-hyperdet}). See also \cite{distributed-entanglement}, where the Cayley hyperdeterminant is rediscovered as \cite[Equations 20 and 21]{distributed-entanglement}.
    
    For the interested reader, let us describe the translation between our notation/terminology and that of quantum information theory. Our base field will now be $\cc$. An element of $\std^{\otimes n}$ (really, of the projective space $\PP(\std^{\otimes n}) \cong \PP^{2^n - 1}$) is called an \textit{$n$-qubit}, and the action of $\SL_2^{\times n}$ is via \textit{stochastic local operations and classical communication} (SLOCC) operators (replacing $\SL_2^{\times n}$ by $\GL_2^{\times n}$ simply amounts to dropping the word ``stochastic''). The space $\std$ is equipped with a basis $\{|0\rangle, |1\rangle\}$, and a cube $\cC = (a, \vec{b}, c, \vec{d}) \in \std^{\otimes 3}$ corresponds to the three-qubit\footnote{Technically, a qubit is required to have norm $1$, so one must rescale $\cC$ by $\sqrt{a^2 + \|\vec{b}\|^2 + c^2 + \|\vec{d}\|^2}$; but this could in theory introduce a singularity when $a^2 + \|\vec{b}\|^2 + c^2 + \|\vec{d}\|^2 = 0$. We will ignore this (important!) point below.}
    \begin{align*}
        a |000\rangle & + b_1 |100\rangle + b_2 |010\rangle + b_3 |001\rangle \\
        + d_1 |011\rangle & + d_2 |101\rangle + d_3 |110\rangle + c |111\rangle.
    \end{align*}
    Here, the bra-ket notation $|ijk\rangle$ means $|i\rangle \otimes |j\rangle \otimes |k\rangle$.
    The state 
    $$\tfrac{1}{\sqrt{2}}(1, \vec{0}, 1, \vec{0}) = \tfrac{1}{\sqrt{2}}\left(|000\rangle + |111\rangle\right)$$
    is known as the \textit{Greenberger–Horne–Zeilinger} (GHZ) state, and the state 
    $$\tfrac{1}{\sqrt{3}} \kappa(0) = \tfrac{1}{\sqrt{3}} (0,\vec{1},0,\vec{0}) = \tfrac{1}{\sqrt{3}}\left(|001\rangle + |010\rangle + |100\rangle\right)$$
    is called the \textit{W} state. These two states are known to represent two very different kinds of quantum entanglement; the reason for this is simply that the Cayley hyperdeterminant of the GHZ state is nonzero, but the Cayley hyperdeterminant of the W state vanishes. Nevertheless, the proof of \cref{prop: stab of kostant} shows that there is a natural \textit{degeneration} of (the SLOCC/$\SL_2^{\times 3}$-equivalence class of) the GHZ state into the W state. (This degeneration is $\kappa(a^2)$, whose norm is $\sqrt{a^4 + 3}$; in particular, it passes through zero when $a = \sqrt[4]{-3}$, and so this one-parameter family is not physical!)
    In fact, this state already appears as \cite[Equation 20]{dur-vidal-cirac}.
\end{remark}

One of the key properties of the Kostant section/companion matrices is that a matrix $A \in \sl_2^\ast$ is conjugate to $\kappa(\det(A))$ if and only if $A$ is regular (i.e., the minimal polynomial of $A$ agrees with its characteristic polynomial), if and only if $A$ is nonzero. We will now prove an analogous result concerning $\kappa: \AA^1 \to \std^{\otimes 3}$. 
\begin{prop}\label{prop: G-orbit kostant codim}
    The $\ld{G}$-orbit of the image of $\kappa$ is a dense open subscheme whose complement has codimension $3$.
\end{prop}
\begin{proof}
    We will use the classification of $\ld{G}$-orbits on $\std^{\otimes 3}$ as in \cite[Example 4.5 in Chapter 14]{gelfand-hyperdet}; see \cref{fig: orbits C2 cubed} for a graph of the seven orbits of $\ld{G}$ on $\std^{\otimes 3}$. Namely, if $\lambda \neq 0$, all elements of $\det^{-1}(\lambda)$ are in a single $\ld{G}$-orbit. (In fact, all elements in the fiber $\det^{-1}(1)$ are in the $\ld{G}$-orbit of $(1, \vec{0}, 1, \vec{0})$.) The $\ld{G}$-orbit of $\det^{-1}(\GG_m)$ is open and dense, and hence is $8$-dimensional; moreover, it agrees with the $\ld{G}$-orbit of $\kappa(\GG_m)$. Next, there is a maximal $\ld{G}$-orbit inside the fiber $\det^{-1}(0)$, given by the orbit of $(0,\vec{0}, 0, \vec{1}) = \kappa(0)$. This orbit is $7$-dimensional, and the largest $\ld{G}$-orbits contained in the complement $\det^{-1}(0) - \ld{G}\cdot \kappa(0)$ have dimension $5$. In particular, the complement of $\ld{G}\cdot \kappa(\AA^1) \subseteq \std^{\otimes 3}$ has dimension $5$, i.e., codimension $8-5=3$.
\end{proof}

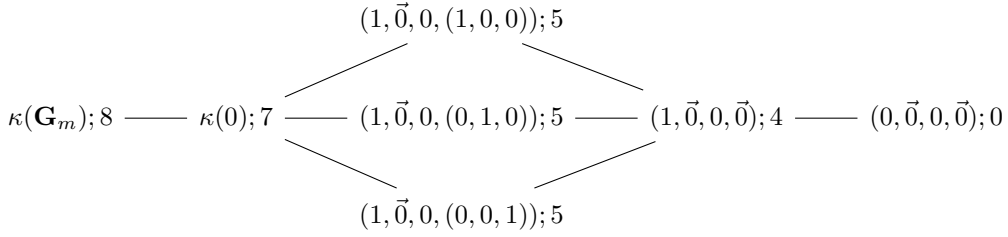
\begin{figure}[H]
\adjustbox{scale=1,center}{%
\begin{tikzcd}
	&& {(1,\vec{0},0,(1,0,0)); 5} \\
	{\kappa(\GG_m); 8} & {\kappa(0); 7} & {(1,\vec{0},0,(0,1,0)); 5} & {(1,\vec{0},0,\vec{0}); 4} & {(0,\vec{0},0,\vec{0}); 0} \\
	&& {(1,\vec{0},0,(0,0,1)); 5}
	\arrow[no head, from=2-1, to=2-2]
	\arrow[no head, from=2-2, to=1-3]
	\arrow[no head, from=2-2, to=2-3]
	\arrow[no head, from=2-2, to=3-3]
	\arrow[no head, from=1-3, to=2-4]
	\arrow[no head, from=2-3, to=2-4]
	\arrow[no head, from=3-3, to=2-4]
	\arrow[no head, from=2-4, to=2-5]
\end{tikzcd}
}
\captionsetup{width=\linewidth}
\caption[]{$\ld{G}$-orbits on $\std^{\otimes 3}$, representatives, and their dimensions (indicated after the semicolon), connected by closure. Note that $\kappa(0) = (0,\vec{0},0,\vec{1})$, and that the $\ld{G}$-orbit of $\kappa(1) = (1,\vec{0},0,\vec{1})$ is the same as the $\ld{G}$-orbit of $(1,\vec{0},1,\vec{0})$.}
\label{fig: orbits C2 cubed}
\end{figure}
\begin{remark}\label{rmk: G-orbit segre}
    As explained in \cite[Example 4.5 in Chapter 14]{gelfand-hyperdet}, the closure of the associated orbits inside $\PP(\std^{\otimes 3}) = \PP^7$ can be described as follows. First, the closure of the generic orbit is $\PP^7$. Next, the closure of the orbit of next smallest dimension is the zero locus of $\det$, which cuts out the dual variety of the Segre embedding $(\PP^1)^{\times 3} \hookrightarrow \PP^7$ (just as the usual determinant for $2 \times 2$-matrices cuts out the quadric $\PP^1 \times \PP^1 \hookrightarrow \PP^3$). The projective orbit associated to $(1,\vec{0},0,(0,1,0))$, say, is cut out inside the locus $\{\det = 0\}$ by the Segre embedding $\PP(\std) \times \PP(\std^{\otimes 2}) = \PP^1 \times \PP^3 \to \PP^7$. Finally, the minimal nonzero orbit is cut out by the Segre embedding $(\PP^1)^{\times 3} \to \PP^7$.
\end{remark}

\begin{remark}\label{rmk: SOn more general}
    More generally, let $\std_n$ denote the standard $n$-dimensional representation of $\SO_n$, so that the symplectic vector space $\std \otimes \std_n$ is equipped with an action of $\SL_2 \times \SO_n$. Using \cite[Section 7]{sato-kimura-prehomogeneous}, one finds that the obvious analogue of the formula for $\det(\cC)$ in \cref{rmk: cayley as actual det} defines a map $\std \otimes \std_n \to\AA^1\mmod (\Z/2)$ which induces an isomorphism
    $$(\std \otimes \std_n)\mmod (\SL_2 \times \SO_n) \cong \AA^1\mmod (\Z/2).$$
    If we allow ourselves a square root of $2$, then \cref{prop: G-orbit kostant codim} admits an analogue in this more general setting (at least if one works over $\cc$): there is a Kostant slice $\kappa: \AA^1\mmod (\Z/2) \to \std \otimes \std_n$ whose $\SL_2 \times \SO_n$-orbit is open and has complement of codimension $3$. Namely, assume $n = 2j$ is even for simplicity (a slight variant of this construction will work for odd $n$), so that without loss of generality, the symmetric bilinear form on $\std_n$ is given by $\begin{psmallmatrix}
        0 & 1\\
        1 & 0
    \end{psmallmatrix}^{\oplus j}$. If $e_1, \cdots, e_{2j}$ is a basis for $\std_n$, let $v_1 = a^2e_1 + e_2$, and let $v_2 = \sum_{i=2}^j (e_{2i-1} + e_{2i})$. Then $\pdb{v_1, v_1} = 2a^2$, $\pdb{v_2, v_2} = 2(j-1)$, and $\pdb{v_1, v_2} = 0$. If $e_1, e_2$ is a basis for $\std$, the Kostant slice sends
    $$\kappa: \AA^1\mmod (\Z/2) \to \std \otimes \std_n, \ a^2 \mapsto \tfrac{1}{\sqrt{2}} e_1 \otimes v_1 + \tfrac{1}{\sqrt{2(j-1)}} e_2 \otimes v_2.$$
    It is easy to check that this map does indeed give a section of $\det$.
    To check that the $\SL_2 \times \SO_n$-orbit of $\kappa$ has complement of codimension $3$, we need an analogue of \cref{rmk: G-orbit segre}. This succumbs to an analysis similar to that of \cite[Chapter 14]{gelfand-hyperdet}. One finds that if $n\geq 5$, the poset of closures of $\SL_2 \times \SO_n$-orbits in $\PP(\std \otimes \std_n) \cong \PP^{2n-1}$ is as shown in \cref{fig: orbits C2 tensor Cn}. The case $n=4$ is ``degenerate'' and one instead gets \cref{fig: orbits C2 cubed}.
    \begin{figure}[H]
    \adjustbox{scale=1,center}{%
    \begin{tikzcd}
	&&& \PP^1 \times \PP^{n-1} \\
	{\PP(\std \otimes \std_n)} & \{\det = 0\} &&& {X \cap (\PP^1 \times \PP^{n-1})} \\
	&& X
	\arrow[no head, from=2-1, to=2-2]
	\arrow[no head, from=2-2, to=3-3]
	\arrow[no head, from=3-3, to=2-5]
	\arrow[no head, from=1-4, to=2-5]
	\arrow[no head, from=2-2, to=1-4]
    \end{tikzcd}
    }
    \captionsetup{width=\linewidth}
    \caption[]{$\SL_2 \times \SO_n$-orbit closures on $\std \otimes \std_n$, connected by closure. The generic orbit is given by the nonvanishing of $\det$. If an element of $\std \otimes \std_n$ is given by $e_1 \otimes v_1 + e_2 \otimes v_2$ with $v_1, v_2\in \std_n$, the subvariety $X$ has codimension $3$, and is cut out by $\begin{psmallmatrix}
        \pdb{v_1, v_1} & \pdb{v_1, v_2} \\
        \pdb{v_2, v_1} & \pdb{v_2, v_2}
    \end{psmallmatrix} = 0$. Moreover, the inclusion $\PP^1 \times \PP^{n-1} \hookrightarrow \PP^{2n-1}$ is the Segre embedding, i.e., is cut out by $v_1 \wedge v_2 = 0\in \wedge^2 \std_n$.}
    \label{fig: orbits C2 tensor Cn}
    \end{figure}
    
    {The motivation for this example comes from attempting to generalize the discussion in \cite[Section 4]{bhargava-composition-i}. Indeed, removing the vertex in the Dynkin diagram of type $\mathrm{D}_{j+2}$ which is connected to the affine root in the extended Dynkin diagram defines a maximal parabolic subgroup $P$ of $\SO_{2j+2}$, and its Levi quotient $L$ is $\SL_2 \times \SO_{2j}$. (When this procedure is applied to the Dynkin diagram of a general semisimple $G$, the resulting Levi subgroup is the centralizer of the $\SL_2$ embedded in $G$ via the highest root of $\g$.) If $U$ denotes the unipotent radical of $P$, then $L$ acts on the vector space $U/[U,U]$ by conjugation, and the Lie bracket on $U$ defines a symplectic form on $U/[U,U]$. With a bit of pain, one can check that $U/[U,U] \cong \std \otimes \std_{2j}$ as a symplectic $L \cong \SL_2 \times \SO_{2j}$-representation.
    
    The same construction with the Dynkin diagram of type $\mathrm{B}_{j+2}$ produces $\SL_2 \times \SO_{2j+1}$ acting on $\std \otimes \std_{2j+1}$. Doing this procedure for the other Dynkin diagrams produces some of the ``vectorial'' examples in one of the columns in \cite[Table 1.5.1]{bzsv}. For example, the type $\mathrm{A}_{n+1}$ Dynkin diagram produces $\GL_n$ acting on $T^\ast(\std_n)$, and I believe the Dynkin diagram for $\mathrm{E}_6$ will produce $\SL_6$ acting on $\wedge^3 \std_6$. One also obtains some representations \textit{not} included by \cite{bzsv}; for instance, the Dynkin diagram of $\G_2$ will produce $\SL_2$ acting on $\Sym^3(\std)$, which is not ``hyperspherical'' (see \cite[Example 5.1.10]{bzsv}). For this last calculation, see \cite[Pages 160-161]{springer-linear-algebraic-groups}.
    I do not yet understand the significance of these observations in the context of relative geometric Langlands.}
\end{remark}

This leads us to the main calculation. It implies, for instance, that the stabilizer (inside $\SL_2^{\times 3}$) of a element of $\std^{\otimes 3}$ with nonvanishing Cayley hyperdeterminant can be identified with $\GG_m^2 \cong \ker(\GG_m^3 \xar{\text{prod}} \GG_m)$.
\begin{prop}\label{prop: stab of kostant}
    Let $\ld{J}$ denote the group scheme over $\sl_2^\ast\mmod \SL_2 \cong \spec \ringcoeff[a^2]$ of regular centralizers for $\SL_2$, so that
    \begin{align*}
        \ld{J} & \cong \spec \ringcoeff[a, \alpha^{\pm 1}, \tfrac{\alpha - \alpha^{-1}}{a}]^{\Z/2} \cong \spec \ringcoeff[a^2, \alpha + \alpha^{-1}, \tfrac{\alpha - \alpha^{-1}}{a}] \\
        & \cong \ker(\Res_{\ringcoeff[a]/\ringcoeff[a^2]} \GG_m \xar{\Nm} \GG_m),
    \end{align*}
    where the action of $\Z/2$ sends $a \mapsto -a$ and $\alpha \mapsto \alpha^{-1}$, and the group structure is such that $\alpha$ is grouplike.
    Then there is an isomorphism
    $$\sl_2^\ast\mmod \SL_2 \times_{\std^{\otimes 3}/\ld{G}} \sl_2^\ast\mmod \SL_2 \cong \ker(\ld{J} \times_{\sl_2^\ast\mmod \SL_2} \ld{J} \times_{\sl_2^\ast\mmod \SL_2} \ld{J} \xar{\mathrm{prod}} \ld{J})$$
    of group schemes over $\sl_2^\ast\mmod \SL_2 = \spec \ringcoeff[a^2]$; of course, this group scheme is in turn isomorphic to $\ld{J} \times_{\sl_2^\ast\mmod \SL_2} \ld{J}$.
\end{prop}
\begin{proof}
    The fiber product on the left identifies with the subgroup of $\sl_2^\ast\mmod \SL_2 \times \ld{G}$ of those $(a^2,\vec{g})$ such that $\vec{g} = (g_1, g_2, g_3)\in \SL_2^{\times 3}$ stabilizes $\kappa(a^2)$. The trick to determining this stabilizer is to use Bhargava's construction from \cref{def: cube and quadratic forms}: if $\vec{g}$ stabilizes a cube $\cC$, it must also stabilize the corresponding triple $\mu(\cC)\in \sl_2^{\ast, \times 3}$ of quadratic forms. 
    
    First, a simple calculation shows that if $a$ is a unit, the triple of matrices
    $$\vec{g} = \left(
    \begin{pmatrix}
        -1 & a^{-1}\\
        a & 1
    \end{pmatrix}, 
    \begin{pmatrix}
        -1 & a^{-1}\\
        a & 1
    \end{pmatrix}, 
    \begin{pmatrix}
        -1 & a^{-1}\\
        a & 1
    \end{pmatrix}
    \right) \in \SL_2^{\times 3}$$
    sends 
    $$\kappa(a^2) \mapsto -4 (a^2, \vec{0}, a^{-1}, \vec{0}).$$
    The triple $\vec{g}$ can be thought of as ``diagonalizing'' $\kappa(a^2)$.
    The stabilizer of the cube $-4 (a^2, \vec{0}, a^{-1}, \vec{0})$ precisely consists of triples of matrices of the form
    \begin{equation}\label{eq: diag stab}
        \left(\begin{pmatrix}
            \alpha_1 & 0\\
            0 & \alpha_1^{-1}
        \end{pmatrix}, \begin{pmatrix}
            \alpha_2 & 0\\
            0 & \alpha_2^{-1}
        \end{pmatrix}, \begin{pmatrix}
            \alpha_3 & 0\\
            0 & \alpha_3^{-1}
        \end{pmatrix}
        \right) \text{ with }\alpha_1\alpha_2\alpha_3 = 1.
    \end{equation}
    For $\alpha \in \GG_m$, let $h(\alpha)$ denote the matrix
    $$h(\alpha) = \frac{1}{2} \begin{pmatrix}
        \alpha + \alpha^{-1} & \frac{\alpha^{-1} - \alpha}{a} \\
        a^2\cdot \frac{\alpha^{-1} - \alpha}{a} & \alpha + \alpha^{-1}
    \end{pmatrix} \in \SL_2.$$
    Conjugating \cref{eq: diag stab} by the element $\vec{g} \in \ld{G}$, we find that the triple $(h(\alpha_1), h(\alpha_2), h(\alpha_3))$ of matrices stabilizes $\kappa(a^2)$ as long as $\alpha_1\alpha_2\alpha_3 = 1$ and $a^2\in \GG_m \subseteq \AA^1$. (See \cite[Section 3.2]{bfm} for a slight variant of this calculation.) 
    Note that the subgroup of such triples is $2$-dimensional, and therefore the associated homogeneous $\ld{G}$-space is $9 - 2 = 7$-dimensional. Using that the $\ld{G}$-orbit of $\kappa(a^2)$ is also $7$-dimensional (e.g., by \cite[Example 4.5 in Chapter 14]{gelfand-hyperdet}), it is not hard to see from this calculation (by a limiting argument for $a \to 0$) that the stabilizer of the family $\kappa(\AA^1) \subseteq \std^{\otimes 3}$ is precisely the claimed group scheme.
\end{proof}
\begin{remark}
    Using \cite[Corollary 3.6.20]{ku-rel-langlands} with $H = \SO_3$ and $G = \PSO_4$, we find the following consequence of \cref{prop: stab of kostant}. Choose one of the factors $\SL_2 \subseteq \SL_2^{\times 3}$, let $\GG_a\subseteq \SL_2$ denote the subgroup of strictly upper-triangular matrices, and let $\psi: \GG_a \to \GG_a$ denote the identity character. Then, there is an isomorphism
    $$(\std^{\otimes 3})^\reg/_\psi \GG_a \cong T^\ast(\SL_2)^\reg$$
    of $\SL_2 \times \SL_2$-varieties.
\end{remark}

\section{The proof, and some remarks}

Before proceeding, let us remind the reader of the definition of the left-hand side of the equivalence of \cref{thm: main}, following \cite[Definition 3.6.1]{ku-rel-langlands}.
\begin{definition}\label{def: Shv-Sat LG/H}
    Let $G$ be a complex reductive group, and let $H\subseteq G$ be a closed subgroup.
    Let $\Shv^{c}_{G\pw{t}}(G\ls{t}/H\ls{t}; \QQ)$ denote the $\infty$-category of $G\pw{t}$-equivariant sheaves of $\QQ$-modules on $G\ls{t}/H\ls{t}$ which are constructible for the orbit stratification on $G\ls{t}/H\ls{t}$.
    There is a natural left-action of the $\E{3}$-monoidal $\infty$-category $\Shv^c_{(G\times G)\pw{t}}(G\ls{t}; \QQ)$ on $\Shv^{c}_{G\pw{t}}(G\ls{t}/H\ls{t}; \QQ)$, and in particular, a left-action of $\Rep(\ld{G})$ by the abelian geometric Satake theorem of \cite{mirkovic-vilonen}.
    Let 
    $$\IC_0\in \Shv^{c}_{G\pw{t}}(G\ls{t}/H\ls{t}; \QQ)$$
    denote the pushforward $i_! \ul{\QQ}$ of the constant sheaf along the inclusion $(G/H)(\cc\pw{t}) \to (G/H)(\cc\ls{t})$.
    Let 
    $$\Shv^{c,\Sat}_{G\pw{t}}(G\ls{t}/H\ls{t}; \QQ) \subseteq \Shv^{c}_{G\pw{t}}(G\ls{t}/H\ls{t}; \QQ)$$
    denote the full subcategory generated by $\IC_0$ under the action of $\Rep(\ld{G})$. If $\ringcoeff$ is any $\QQ$-algebra, base-changing along the unit map defines the $\infty$-category $\Shv^{c,\Sat}_{G\pw{t}}(G\ls{t}/H\ls{t}; \ringcoeff)$.
\end{definition}
\begin{proof}[Proof of \cref{thm: main}]
    It suffices to verify conditions (a) and (b) of \cite[Theorem 3.6.4]{ku-rel-langlands}, which gives a criterion for establishing an equivalence of $\ringcoeff$-linear $\infty$-categories of the form
    $$\Shv^{c,\Sat}_{G\pw{t}}(G\ls{t}/H\ls{t}; \ringcoeff) \simeq \Perf(\sh^{1/2} \ld{M}/\ld{G}).$$
    The map $\kappa$ is given by the map $\sl_2^\ast(2)\mmod \SL_2 \to \std^{\otimes 3}(4,\vec{2},-2,\vec{0})$ from \cref{cstr: kostant}. 
    For condition (a) of \cite[Theorem 3.6.4]{ku-rel-langlands}, we need to show that if $\ld{J}_X = \sl_2^\ast(2)\mmod \SL_2 \times_{\std^{\otimes 3}(4,\vec{2},-2,\vec{0})/\ld{G}} \sl_2^\ast(2)\mmod \SL_2$, the ring of regular functions on the quotient $(\sl_2^\ast(2)\mmod \SL_2 \times \ld{G})/\ld{J}_X$ is isomorphic (as a graded algebra) to $\co_{\std^{\otimes 3}(4,\vec{2},-2,\vec{0})}$. The quotient $(\sl_2^\ast(2)\mmod \SL_2 \times \ld{G})/\ld{J}_X$ identifies with the $\ld{G}$-orbit of the image of $\kappa$, which has complement of codimension $3$ in $\std^{\otimes 3}$ by \cref{prop: G-orbit kostant codim}; therefore, the algebraic Hartogs theorem implies that there is a graded isomorphism $\co_{(\sl_2^\ast(2)\mmod \SL_2 \times \ld{G})/\ld{J}_X} \cong \co_{\std^{\otimes 3}(4,\vec{2},-2,\vec{0})}$.

    For condition (b) of \cite[Theorem 3.6.4]{ku-rel-langlands}, we need to check that there is an isomorphism
    $$\ld{J}_X \cong \spec \H^{\PGL_2}_\ast(\Omega(\PGL_2^{\times 3}/\PGL_2^\mathrm{diag}); \ringcoeff)$$
    of graded group schemes over $\sl_2^\ast(2)\mmod \SL_2 \cong \spec \H^\ast_{\PGL_2}(\ast; \ringcoeff)$. There is an isomorphism 
    \begin{equation}\label{eq: reg centr PGL2}
        \spec \H^{\PGL_2}_\ast(\Omega \PGL_2; \ringcoeff) \cong \spec \ringcoeff[a, \alpha^{\pm 1}, \tfrac{\alpha - \alpha^{-1}}{a}]^{\Z/2} \cong \ld{J},
    \end{equation}
    and the action of the $\Z/2$ on the middle term sends $a\mapsto -a$ and $\alpha \mapsto \alpha^{-1}$. This is proved, e.g., in \cite{bfm}. (As in \cref{prop: stab of kostant}, $\ld{J}$ denotes the group scheme over $\sl_2^\ast\mmod \SL_2$ of regular centralizers for $\SL_2$.) The K\"unneth theorem implies that there is an isomorphism 
    $$\spec \H^{\PGL_2}_\ast(\Omega (\PGL_2^{\times 3}); \ringcoeff) \cong \spec \ringcoeff[a, \alpha_i^{\pm 1}, \tfrac{\alpha_i - \alpha_i^{-1}}{a} | 1\leq i \leq 3]^{\Z/2},$$
    and the fiber sequence
    $$\PGL_2^\mathrm{diag} \to \PGL_2^{\times 3} \to \PGL_2^{\times 3}/\PGL_2^\mathrm{diag}$$
    implies that
    $$\spec \H^{\PGL_2}_\ast(\Omega (\PGL_2^{\times 3}/\PGL_2^\mathrm{diag}); \ringcoeff) \cong \ker(\ld{J} \times_{\sl_2^\ast\mmod \SL_2} \ld{J} \times_{\sl_2^\ast\mmod \SL_2} \ld{J} \xar{\mathrm{prod}} \ld{J}).$$
    The desired isomorphism now follows from this observation and \cref{prop: stab of kostant}.
\end{proof}
\begin{remark}
    Let $\ringcoeff[\hbar] = \H^{\ast}_{S^1_\rot}(\ast; \ringcoeff)$, so that $\hbar$ lives in weight $-2$.
    Let $\cd_\hbar(\std^{\ast, \otimes 3})$ denote the Weyl algebra of $\std^{\otimes 3}$, so that it is generated over $\ringcoeff[\hbar]$ by $\Sym_{\ringcoeff}(\std^{\ast,\otimes 3})$, where the commutation relation is
    $$[v_1 \otimes w_1 \otimes u_1, v_2 \otimes w_2 \otimes u_2] = \hbar \pdb{v_1, v_2} \pdb{w_1, w_2} \pdb{u_1, u_2}.$$
    The algebra $\cd_\hbar(\std^{\ast, \otimes 3})$ acquires a grading coming from the graded symplectic vector space $\std^{\otimes 3}(4, \vec{2}, -2, \vec{0})$, and it can be characterized as the unique graded $\ringcoeff[\hbar]$-algebra such that the induced Poisson bracket on $\cd_\hbar(\std^{\ast, \otimes 3}) \otimes_{\ringcoeff[\hbar]} \ringcoeff$ is the one coming from the symplectic form on $\std^{\otimes 3}$. Using this fact, one can show that there is a $\ringcoeff[\hbar]$-linear equivalence
    $$\Shv_{\PGL_2^{\times 3}\pw{t} \rtimes \GG_m^\rot}^{c,\Sat}(\PGL_2^{\times 3}\ls{t}/\PGL_2^\mathrm{diag}\ls{t}; \ringcoeff) \simeq \sh^{1/2}\left( \cd_\hbar(\std^{\ast, \otimes 3})\modc^{\SL_2(-2\rho)^{\times 3}} \right)$$
    which extends \cref{thm: main}.
\end{remark}
\begin{remark}
    Specializing \cite[Remark 3.6.12]{ku-rel-langlands} to the present case, one can argue as in \cite[Theorem 12.5.3]{arinkin-gaitsgory-singsupp} to show that an object of $\Shv_{\PGL_2^{\times 3}\pw{t}}^{c,\Sat}(\PGL_2^{\times 3}\ls{t}/\PGL_2^\mathrm{diag}\ls{t}; \ringcoeff)$ is compact if and only if its image under the equivalence of \cref{thm: main} is set-theoretically supported on the vanishing locus of the Cayley hyperdeterminant.
\end{remark}
\begin{remark}
    We have already verified most of \cite[Conjecture 3.5.11]{ku-rel-langlands} in the case of the spherical $\PGL_2^{\times 3}$-variety $\PGL_2^{\times 3}/\PGL_2^\mathrm{diag}$. It remains to check the final part, which is the same as \cite[Conjecture 1.1.1]{finkelberg-ginzburg-travkin}. This states that if $B\subseteq \PGL_2$ is a Borel subgroup, the set of $B^{\times 3}$-orbit closures on $\PGL_2^{\times 3}/\PGL_2^\mathrm{diag}$ agrees with the set of irreducible components of $\std^{\otimes 3} \times_{\ld{\fr{b}}^{\ast, \times 3}} \{0\}$. (We will not check that this bijection is equivariant for the Weyl group.) The set of $B^{\times 3}$-orbits on $\PGL_2^{\times 3}/\PGL_2^\mathrm{diag}$ can be identified with the set of $\PGL_2$-orbits on $\PGL_2^{\times 3}/B^{\times 3} \cong (\PP^1)^{\times 3}$. Direct computation verifies that there are five such orbits. 

    On the other hand, the formulas of \cref{def: cube and quadratic forms} show that the fiber product $\std^{\otimes 3} \times_{\ld{\fr{b}}^{\ast, \times 3}} \{0\}$ consists of those cubes $(a, \vec{b}, c, \vec{d})$ such that
    \begin{align*}
        a d_1 & = b_2 b_3, \\
        a d_2 & = b_1 b_3, \\
        a d_3 & = b_2 b_1, \\
        ac + b_1 d_1 & = b_2 d_2 + b_3 d_3, \\
        ac + b_2 d_2 & = b_1 d_1 + b_3 d_3, \\
        ac + b_3 d_3 & = b_1 d_1 + b_2 d_2.
    \end{align*}
    The primary decomposition of the ideal in $\co_{\std^{\otimes 3}}$ cut out by these equations is given by the intersection of the following five prime ideals:
    \begin{align*}
        (a, b_1, b_3, d_2), & (a, b_1, b_2, b_3), \\
        (a, b_1, b_2, d_3), & (a, b_2, b_3, d_1), \\
        (b_2 d_2 - b_3 d_3, b_1 d_1 - b_3 d_3, b_3 c - d_1 d_2, & b_2 c - d_1 d_3, b_1 c - d_2 d_3, a c - b_3 d_3, \\
        b_2 b_3 - a d_1, b_1 b_3 - a d_2, & b_1 b_2 - a d_3, a d_1 d_2 - b_3^2 d_3).
    \end{align*}
    In particular, the fiber product $\std^{\otimes 3} \times_{\ld{\fr{b}}^{\ast, \times 3}} \{0\}$ has five irreducible components, as desired.
\end{remark}
\begin{remark}
    The proof of \cref{thm: G2 main} is exactly the same as the proof of \cref{thm: main} above. Indeed, one only needs to observe the following (which we will prove momentarily).
    \begin{lemma}\label{lem: pso8 mod g2}
        There is a homotopy equivalence $\PSO_8/\G_2 \simeq \RP^7 \times \RP^7$.
    \end{lemma}
    The replacement of \cref{eq: reg centr PGL2} is given by \cite[Proposition 4.8.6]{ku-rel-langlands}, which gives an isomorphism
    $$\spec \H^{\G_2}_\ast(\Omega \RP^7; \ringcoeff) \cong \spec \ringcoeff[a, b, \alpha^{\pm 1}, \tfrac{\alpha - \alpha^{-1}}{a}]^{\Z/2}$$
    where $a$ is in weight $-6$ and $b$ is in weight $-4$; the argument of \cref{thm: main} then proves \cref{thm: G2 main}.
    \begin{proof}[Proof of \cref{lem: pso8 mod g2}]
        This follows from the claim that $\Spin_8/\G_2 \simeq S^7 \times S^7$.
        Perhaps the most ``conceptual'' way to see this is as follows. Using triality, one can identify $\Spin_8$ with the subgroup of $\SO_8^{\times 3}$ of those triples $(A_1, A_2, A_3)$ such that $A_1(x_1) A_2(x_2) = A_3(x_1 x_2)$ for octonions $x_1,x_2$. Under this presentation, $\G_2$ corresponds to the subgroup where $A_1 = A_2 = A_3$. The subgroups where $A_1 = A_3$ (resp. $A_2 = A_3$) are both isomorphic to $\Spin_7$; these are sometimes denoted $\Spin^\pm_7$. The action of $\Spin_8$ on $S^7 \times S^7$ sends $(x,y) \mapsto (A_1 x, A_2 y)$; one can check that this is transitive, and that the stabilizer of the point $(1,1)$ is precisely $\Spin^+_7 \cap \Spin^-_7 \cong \G_2$.
    
        That there is an equivalence $\Spin_8/\G_2 \simeq S^7 \times S^7$ at the level of cohomology with $\Z[1/2]$-coefficients and mod $2$ coefficients, at least, is much simpler. On cohomology with $\Z[1/2]$-coefficients, the map $\G_2 \to \Spin_8$ is given by the map $\Z[1/2, p_1, p_2, p_3, c_4] \to \Z[1/2, c_2, c_6]$ sending
        $$p_1 \mapsto 2c_2, \ p_2 \mapsto c_2^2, \ p_3 \mapsto c_6, \ c_4 \mapsto 0.$$
        Indeed, the map $\G_2 \to \Spin_8$ induces a map on maximal tori, whose effect on cohomology is the map $\Z[1/2, x_1, x_2, x_3, x_4] \to \Z[1/2, y_1, y_2]$ sending 
        $$x_1\mapsto 0, \ x_2 \mapsto y_1, \ x_3\mapsto y_2, \ x_4\mapsto -(y_1 + y_2).$$
        Expressing $p_j$ as the $j$th elementary symmetric polynomial in $x_1^2, \cdots, x_4^2$ and $c_4 = x_1x_2x_3x_4$, and using that $c_2 = y_1^2 + y_2^2 + y_1 y_2$ and $c_6 = y_1^2 y_2^2 (y_1+y_2)^2$ gives the desired claim.
        The Serre spectral sequence for the fibration $\Spin_8/\G_2 \to B\G_2 \to B\Spin_8$ implies that 
        $$\H^\ast(\Spin_8/\G_2; \Z[1/2]) \cong \Z[1/2, \sigma(4p_2 - p_1^2), \sigma(c_4)]/(\sigma(4p_2 - p_1^2)^2, \sigma(c_4)^2),$$
        where $\sigma(4p_2 - p_1^2)$ and $\sigma(c_4)$ both live in (homological) weight $-7$. This is precisely the cohomology of $S^7 \times S^7$, as desired.

        The story in mod $2$ cohomology is similar. Namely, \cite{quillen-mod-2-coh-spinor} tells us that $\H^\ast_{\Spin_8}(\ast; \FF_2)$ is isomorphic to the polynomial algebra $\FF_2[w_4, w_6, w_7, w_8, \epsilon]$, where $\epsilon$ lives in cohomological degree $8$; the map $\H^\ast_{\Spin_8}(\ast; \FF_2) \to \H^\ast_{\G_2}(\ast; \FF_2) \cong \FF_2[w_4, w_6, w_7]$ sends $w_8,\epsilon\mapsto 0$. It follows that 
        $$\H^\ast(\Spin_8/\G_2; \FF_2) \cong \FF_2[\sigma(w_8), \sigma(\epsilon)]/(\sigma(w_8)^2, \sigma(\epsilon)^2),$$
        where $\sigma(w_8)$ and $\sigma(\epsilon)$ both live in (homological) weight $-7$. This is again the cohomology of $S^7 \times S^7$, as desired.
    \end{proof}
\end{remark}
\begin{remark}
    The $\infty$-category $\Shv_{\PGL_2^{\times 3}\pw{t}}^{c,\Sat}(\PGL_2^{\times 3}\ls{t}/\PGL_2^\mathrm{diag}\ls{t}; \ringcoeff)$ admits a natural action of the symmetric group $\Sigma_3$. Under the equivalence of \cref{thm: main}, this corresponds to the $\Sigma_3$-action on $\std^{\otimes 3}/\SL_2^{\times 3}$ which permutes the tensor factors. As explained in \cite[Remark 3.6.7]{ku-rel-langlands}, this $\Sigma_3$-action can be understood as an analogue of the Gelfand-Graev action (which, for connected semisimple $G$, gives an action of the Weyl group of $G$ on the affine closure of $T^\ast(G/N)$).
\end{remark}
\begin{remark}
    \cref{rmk: sympl on C2 cubed} guarantees that the equivalence of \cref{thm: main} is compatible with the action of the spherical Hecke category 
    $$\Shv_{(\PGL_2^{\times 3} \times \PGL_2^{\times 3})\pw{t}}^{c,\Sat}(\PGL_2^{\times 3}\ls{t}; \ringcoeff) \simeq \Perf^\sh(\sl_2^{\ast, \times 3}(2-2\rho)/\SL_2^{\times 3}(-2\rho)).$$
    Namely, there is a commutative diagram
    $$\hspace*{-0.5cm}\xymatrix{
    \Shv_{(\PGL_2^{\times 3} \times \PGL_2^{\times 3})\pw{t}}^{c,\Sat}(\PGL_2^{\times 3}\ls{t}; \ringcoeff) \ar[d]^-{\text{act on }\IC_0} \ar[rr]^-\sim_-{\text{\cite{bf-derived-satake}}} & & \Perf^\sh(\sl_2^{\ast, \times 3}(2-2\rho)/\SL_2^{\times 3}(-2\rho)) \ar[d]^-{\mu^\ast} \\
    \Shv_{\PGL_2^{\times 3}\pw{t}}^{c,\Sat}(\PGL_2^{\times 3}\ls{t}/\PGL_2^\mathrm{diag}\ls{t}; \ringcoeff) \ar[rr]^-\sim_-{\text{\cref{thm: main}}} & & \Perf^\sh(\std^{\otimes 3}(4,\vec{2},-2,\vec{0})/\SL_2^{\times 3}(-2\rho)),
    }$$
    where $\mu^\ast$ is given by pullback along the moment map for the Hamiltonian $\SL_2^{\times 3}$-action on $\std^{\otimes 3}$.

    Let us also note that taking cohomology (i.e., pushforward to a point) defines a functor
    $$\Shv_{\PGL_2^{\times 3}\pw{t}}^{c,\Sat}(\PGL_2^{\times 3}\ls{t}/\PGL_2^\mathrm{diag}\ls{t}; \ringcoeff) \to \Shv_{\PGL_2^{\times 3}\pw{t}}(\ast; \ringcoeff),$$
    which, as discussed in \cite[Remark 3.5.10]{ku-rel-langlands}, factors through the functor $\Shv_{\PGL_2^{\mathrm{diag}}\pw{t}}(\ast; \ringcoeff) \to \Shv_{\PGL_2^{\times 3}\pw{t}}(\ast; \ringcoeff)$. Under \cref{thm: main} and the equivalence $\Shv_{\PGL_2^{\mathrm{diag}}\pw{t}}(\ast; \ringcoeff) \simeq \Perf^\sh(\sl_2^\ast(2)\mmod \SL_2)$, there is a commutative diagram
    $$\hspace*{-0.5cm}
    \xymatrix{
    \Shv_{\PGL_2^{\times 3}\pw{t}}^{c,\Sat}(\PGL_2^{\times 3}\ls{t}/\PGL_2^\mathrm{diag}\ls{t}; \ringcoeff) \ar[d]_-{\text{cohomology}} \ar[rr]^-\sim_{\text{\cref{thm: main}}} & & \Perf^\sh(\std^{\otimes 3}(4,\vec{2},-2,\vec{0})/\SL_2^{\times 3}(-2\rho)) \ar[d]^-{\kappa^\ast} \\
    \Shv_{\PGL_2^{\mathrm{diag}}\pw{t}}(\ast; \ringcoeff) \ar[rr]_-\sim & & \Perf^\sh(\sl_2^\ast(2)\mmod \SL_2),
    }$$
    where $\kappa$ is the Kostant slice of \cref{cstr: kostant}.
\end{remark}
\begin{remark}
    \cref{thm: main} does not need the full strength of optimality in the sense \cite[Hypothesis 3.5.2]{ku-rel-langlands}. Indeed, the first and second assumptions in \cite[Hypothesis 3.5.2]{ku-rel-langlands} are included to ensure formality of the algebra from \cite[Equation 16 in the proof of Theorem 3.6.4]{ku-rel-langlands}. However, as in \cite[Remark 3.2.22]{ku-rel-langlands}, the formality of this algebra is \textit{guaranteed} in our case: since \cref{thm: main} shows that the homotopy of the algebra in question is $\co_{\std^{\otimes 3}(4,\vec{2},-2,\vec{0})}$, i.e., is polynomial on classes in even weights. This algebra admits an $\E{3}$-structure (essentially from factorization; see, e.g., \cite[Proposition 16.1.4]{bzsv}), and is therefore automatically formal by \cite[Lemma 2.1.9]{ku-rel-langlands}. Note, however, that since $\Ind^{\Spin_8}_{\SL_2^{\times 3}}(\std^{\otimes 3} \oplus \AA^1)$ is not an affine space, this argument does not go through in the case of \cref{thm: G2 main} to prove formality of the algebra from \cite[Equation 16 in the proof of Theorem 3.6.4]{ku-rel-langlands}.
\end{remark}

\begin{remark}\label{rmk: so8 minimal nilpotent}
    As in the introduction, let $\ol{\co} := \ol{\co_\min}(\fr{so}_8)$ denote the closure of the minimal nilpotent coadjoint orbit of $\fr{so}_8$. This is a $10$-dimensional scheme with a canonical action of $\SO_8$, and hence a canonical action of $\SL_2^{\times 4} \subseteq \SO_8$.
    By \cite[Proposition 3.18]{jia-affine-closure}, one can identify $\ol{\co}$ with the Hamiltonian reduction at zero of the symplectic vector space $(\AA^2)^{\otimes 3} \otimes \AA^2$ with respect to the Hamiltonian $\SL_2$-action coming from the final factor. That is, if $\mu: (\AA^2)^{\otimes 3} \otimes \AA^2 \to \sl_2^\ast$ denotes the moment map for this Hamiltonian $\SL_2$-action, then 
    $$\mu^{-1}(0)/\SL_2 \cong \ol{\co}$$
    over $\cc$. Another way of saying this is that $\ol{\co}$ is the quotient of $(\AA^2)^{\otimes 3} \otimes \AA^2 \cong (\AA^2)^{\otimes 3} \oplus (\AA^2)^{\otimes 3}$ by the action of the symplectic groupoid $T^\ast \SL_2$.
    
    Using this description of $\ol{\co}$ and our calculations in \cref{prop: stab of kostant}, it is not difficult to show that there is an isomorphism
    $$\ol{\co} \cong \ol{(\SL_2^{\times 4} \times \sl_2^\ast\mmod \SL_2)/\ld{J}_{0,4}},$$
    where the overline on the right-hand side denotes the affine closure, and $\ld{J}_{0,4}$ is the closed subgroup scheme of $\SL_2^{\times 4} \times \sl_2^\ast\mmod \SL_2$ defined by
    $$\ld{J}_{0,4} = \ker(\ld{J} \times_{\sl_2^\ast\mmod \SL_2} \ld{J} \times_{\sl_2^\ast\mmod \SL_2} \ld{J} \times_{\sl_2^\ast\mmod \SL_2} \ld{J} \xar{\mathrm{prod}} \ld{J}).$$
    In particular, by arguing as in \cref{thm: main}, it follows that if the $\PGL_2^{\times 4}\pw{t}$-action on $\PGL_2^{\times 4}\ls{t}/\PGL_2^\mathrm{diag}\ls{t}$ is optimal in the sense of \cite[Hypothesis 3.5.2]{ku-rel-langlands}, there is an equivalence
    $$\Shv_{\PGL_2^{\times 4}\pw{t}}^{c,\Sat}(\PGL_2^{\times 4}\ls{t}/\PGL_2^\mathrm{diag}\ls{t}); \cc) \simeq \Perf^\sh(\ol{\co}/\SL_2^{\times 4})$$
    where the shearing is taken with respect to a certain grading on the stack $\ol{\co}/\SL_2^{\times 4}$.
\end{remark}

\begin{remark}
    There is also an analogue of \cref{prop: stab of kostant} for the $\SL_2 \times \SL_3 \times \SL_3$-representation $\AA^2 \otimes \AA^3 \otimes \AA^3$, studied in \cite{bhargava-composition-ii}. Namely, it turns out that:
    \begin{itemize}
        \item The invariant quotient $(\AA^2 \otimes \AA^3 \otimes \AA^3)\mmod (\SL_2 \times \SL_3 \times \SL_3)$ is isomorphic to $\AA^1$. If one views $\AA^2 \otimes \AA^3 \otimes \AA^3$ as the space of pairs $(M, N)$ of $3\times 3$-matrices, the invariant quotient map is given by
        $$\AA^2 \otimes \AA^3 \otimes \AA^3 \to \AA^1, \ (M, N)\mapsto \Delta(\det(Mx - Ny)),$$
        where $\Delta$ denotes the discriminant of a binary cubic form.
        \item There is a map $\kappa: \AA^1 \to \AA^2 \otimes \AA^3 \otimes \AA^3$ which gives a section of the above invariant quotient map, and its $\SL_2 \times \SL_3 \times \SL_3$-orbit has complement of codimension $2$. Moreover, its stabilizer is given by the group scheme
        $$\AA^1 \times_{(\AA^2 \otimes \AA^3 \otimes \AA^3)/(\SL_2 \times \SL_3 \times \SL_3)} \AA^1 \cong \spec \left(\H^{\PGL_3}_\ast(\Omega \PGL_3; \ringcoeff) \otimes_{\ringcoeff[c_2, c_3]} \ringcoeff[c_3]\right)$$
        over $\AA^1 \cong \spec \ringcoeff[c_3]$.  Here, we have identified $\H^\ast_{\PGL_3}(\ast; \ringcoeff) \cong \ringcoeff[c_2, c_3]$, and the map $\ringcoeff[c_2, c_3] \to \ringcoeff[c_3]$ sends $c_2\mapsto 0$. This group scheme has relative dimension $2 = \rank(\PGL_3)$, and can be identified with the group of matrices of the form
        $$\left\{ \left. \begin{psmallmatrix}
        a & b & c\\
        cc_3 & a & b \\
        bc_3 & cc_3 & a
        \end{psmallmatrix} \right| a^3 + b^3 c_3 + c^3 c_3^2 - 3 a b c c_3 = 1 \right\} \subseteq \SL_3 \times \AA^1.$$
        In fact, one can more generally identify 
        \begin{multline*}
            \spec\left(\H^{\PGL_n}_\ast(\Omega \PGL_n; \ringcoeff) \otimes_{\ringcoeff[c_2, \cdots, c_n]} \ringcoeff[c_n]\right)
            \cong \ker(\Res_{\ringcoeff[c_n^{1/n}]/\ringcoeff[c_n]}(\GG_m) \xar{\mathrm{norm}} \GG_m) \subseteq \SL_n \times \AA^1.
        \end{multline*}
    \end{itemize}
    This is closely related to \cite[Theorem 2]{bhargava-composition-ii}. I hope to use the above observations to explore relationships to relative geometric Langlands in future work.
\end{remark}

\begin{remark}[Gan-Gross-Prasad]\label{rmk: ggp generalization}
    As mentioned in the introduction, \cref{thm: main} can be regarded as a special case of the geometrized analogue of the Gan-Gross-Prasad period (or at least a period isogenous to it), which describes transfer along $\SO_m \subseteq \SO_{m+1}$ when $m=3$.  In general, transfer along $\SO_m \subseteq \SO_{m+1}$ is described by the Hamiltonian variety which is dual, in the sense of \cite{bzsv}, to the spherical $\SO_m \times \SO_{m+1}$-variety $(\SO_m \times \SO_{m+1})/\SO_m^\mathrm{diag}$. 
    In \cite[Example 3.6.24]{ku-rel-langlands}, we described how when $m=2n$, one can obtain the Hamiltonian variety $\Hom(\std_{2j}, \std_{2n})$ which is dual to $(\SO_{2n} \times \SO_{2n+1})/\SO_{2n}^\mathrm{diag}$ via the regular centralizer group schemes $\ld{J}_{\SO_{2n}}$ and $\ld{J}_{\Sp_{2n}}$. Following the philosophy of \cite[Example 3.6.24]{ku-rel-langlands}, let us describe how to obtain the Hamiltonian variety $\Hom(\std_{2n}, \std_{2n-2})$ which is dual to $(\SO_{2n-1} \times \SO_{2n})/\SO_{2n-1}^\mathrm{diag}$ via the regular centralizer group scheme. When $n=2$, this gives an ``alternative approach'' to the results of this article (but it is really a rephrasing of the same argument used above).

    Namely, consider $H = \SO_{2n-1} \subseteq \SO_{2n} = G$, so that $\ld{H} = \Sp_{2n-2}$ and $\ld{G} = \SO_{2n}$. In this case, we claim that the Hamiltonian scheme $\ld{\cM}^\ddag$ from \cite[Construction 3.6.19]{ku-rel-langlands}, defined as the affine closure of $(\ld{G} \times \ld{H} \times \ld{\fr{h}}^\ast\mmod \ld{H})/(\ld{J}_{\ld{G}} \times_{\ld{\g}^\ast\mmod \ld{G}} \ld{\fr{h}}^\ast\mmod \ld{H})$, can be identified with $\Hom(\std_{2n}, \std_{2n-2})$. Again, as in \cite[Example 3.6.24]{ku-rel-langlands}, let us just describe the ``Kostant section'' $\kappa: \ld{\fr{h}}^\ast\mmod \ld{H} \to \Hom(\std_{2n}, \std_{2n-2})$. The desired map $\ld{\cM}^\ddag \to \Hom(\std_{2n}, \std_{2n-2})$ is then obtained using the $\ld{G} \times \ld{H}$-action:
    $$\xymatrix{
    \ld{G} \times \ld{H} \times \ld{\fr{h}}^\ast\mmod \ld{H} \ar[r]^-\kappa \ar[d] & \Hom(\std_{2n}, \std_{2n-2}); \\
    (\ld{G} \times \ld{H} \times \ld{\fr{h}}^\ast\mmod \ld{H})/(\ld{J}_{\ld{G}} \times_{\ld{\g}^\ast\mmod \ld{G}} \ld{\fr{h}}^\ast\mmod \ld{H}) \ar@{-->}[ur]_-\exists &
    }$$
    the dotted map factors through the affine closure of the source, and thereupon induces an isomorphism.

    Recall from \cite[Example 3.2.14]{ku-rel-langlands} that:
    \begin{itemize}
        \item We may identify 
        $$\ld{\fr{h}}^\ast\mmod \ld{H} \cong \spec \ringcoeff[p_1, \cdots, p_{n-1}] \cong \spec \H^\ast_{\SO_{2n-1}}(\ast; \ringcoeff),$$
        and the regular centralizer $\ld{J}_{\ld{H}}$ is the group scheme whose fiber over $\vec{p} := (p_1, \cdots, p_{n-1})$ is the subgroup of those units $f(t) \in \ringcoeff[t]/(t^{2n-2} + p_1 t^{2n-4} + \cdots + p_{n-1})$ such that $f(t)^{-1} = f(-t)$.
        Observe that $\ringcoeff[t]/(t^{2n-2} + p_1 t^{2n-4} + \cdots + p_{n-1})$ admits the structure of a symplectic vector space: the symplectic pairing sends 
        $$(f,g) \mapsto \text{coefficient of }t^{2n-3}\text{ in }f(t)g(-t).$$
        \item We may identify 
        $$\ld{\g}^\ast\mmod \ld{G} \cong \spec \ringcoeff[p_1, \cdots, p_{n-1}, c_n] \cong \spec \H^\ast_{\SO_{2n}}(\ast; \ringcoeff),$$
        and the regular centralizer $\ld{J}_{\ld{G}}$ is the group scheme whose fiber over $(\vec{p}, c_n) := (p_1, \cdots, p_{n-1}, c_n)$ is the subgroup of those units $f(t,v) \in \ringcoeff[t, v]/(tv-c_n, t^{2n-2} + p_1 t^{2n-4} + \cdots + p_{n-1} + v^2)$ such that $f(t,v)^{-1} = f(-t,-v)$. Observe that $\ringcoeff[t, v]/(tv-c_n, t^{2n-2} + p_1 t^{2n-4} + \cdots + p_{n-1} + v^2)$ admits the structure of a quadratic vector space: the associated symmetric bilinear form sends
        $$(f,g) \mapsto \text{coefficient of }t^{2n-2}\text{ in }f(t,v)g(-t,-v).$$
    \end{itemize}
    The map $\pi: \ld{\fr{h}}^\ast \mmod \ld{H} \to \ld{\g}^\ast\mmod \ld{G}$ is induced by the map
    $$\pi: \ringcoeff[p_1, \cdots, p_{n-1}, c_n] \twoheadrightarrow \ringcoeff[p_1, \cdots, p_{n-1}], \ c_n\mapsto 0.$$
    The map $\pi$ induces a map of $\ringcoeff$-vector spaces
    \begin{multline*}
        \varphi_{\vec{p}}: \ringcoeff[t, v]/(tv, t^{2n-2} + p_1 t^{2n-4} + \cdots + p_{n-1} + v^2) \\
        \cong \ringcoeff[t, v]/(tv-c_n, t^{2n-2} + p_1 t^{2n-4} + \cdots + p_{n-1} + v^2) \otimes_{\co_{\ld{\g}^\ast\mmod \ld{G}}} \co_{\ld{\fr{h}}^\ast \mmod \ld{H}} \\
        \to \ringcoeff[t]/(t^{2n-2} + p_1 t^{2n-4} + \cdots + p_{n-1})
    \end{multline*}
    sending $v\mapsto 0$.
    In other words, this is a linear map $\varphi_{\vec{p}}: \std_{2n} \to \std_{2n-2}$. That is, $\pi$ induces a map
    $$\kappa: \ld{\fr{h}}^\ast \mmod \ld{H} \cong \spec \ringcoeff[p_1, \cdots, p_{n-1}] \to \Hom(\std_{2n}, \std_{2n-2}), \ \vec{p} \mapsto \varphi_{\vec{p}};$$
    this is the desired Kostant section.
\end{remark}

\begin{remark}
    Now let $k = \QQ_2(\zeta_8)$.
    The theory of $2$-compact groups as studied, e.g., in \cite{andersen-grodal}, suggests viewing the Dwyer-Wilkerson space $\DW_3$ from \cite{dwyer-wilkerson-2-compact} as an analogue of the groups $\SO_3 \cong \PGL_2$ and $\G_2$; see \cref{table: analogies}. The $2$-complete space $\DW_3$ is equipped with an $\E{1}$-structure, and it has finite mod $2$ cohomology. It is therefore natural to ask whether there is an analogue of \cref{thm: main} and \cref{thm: G2 main}, where $\PGL_2$ and $\G_2$ are replaced by $\DW_3$; this is closely related to \cite[Appendix C(s)]{ku-rel-langlands}. 
    \begin{table}[h]
    \begin{tabular}{ |c|c|c|c|c|c| } 
    \hline
    Group & Rank & Dimension & $\FF_2$-cohomology of $BG$ & Weyl group \\
    $G_n$ & $n$ & $(2^{n+1} - 1)n$ & $\widehat{\Sym}^\ast(\FF_2^{n+1}(-1))^{\GL_{n+1}(\FF_2)}$ & $\Z/2 \times \GL_n(\FF_2)$ \\
    \hline
    \hline
    $\PGL_2$ & $1$ & $3$ & $\FF_2\pw{w_2, w_3}$ & $\Z/2$ \\
    $\G_2$ & $2$ & $14$ & $\FF_2\pw{w_4, w_6, w_7}$ & $\Z/2 \times \Sigma_3$ \\
    $\DW_3$ & $3$ & $45$ & $\FF_2\pw{w_8, w_{12}, w_{14}, w_{15}}$ & $\Z/2 \times \mathrm{PSL}_2(\FF_7)$ \\
    \hline
    \end{tabular}
    \vspace{1cm}
    \caption{Analogies between the ($2$-compact) groups $\PGL_2 = \SO_3$, $\G_2$, and $\DW_3$; all of these are Poincar\'e duality complexes of dimension indicated in the third column. Here, $w_n$ denotes the $n$th Stiefel-Whitney class, and the ring in the fourth column is known as the algebra of rank $n+1$ Dickson invariants. Note, also, that the Weyl group of $\DW_3$ is called $G_{24}$ in the Shephard-Todd classification.}
    \label{table: analogies}
    \end{table}

    It is difficult to answer this question since the representation theory of $\DW_3$ is not well-understood. For instance, one can ask the somewhat outrageous question of whether there is a $2$-compact group $G$ with an $\E{1}$-map $\DW_3 \to G$ such that $G/\DW_3$ is the $2$-completion of a framed $30$-manifold with Kervaire invariant one. (See \cite{jones-30-manifold} for a construction of such a $30$-manifold.) This desideratum is analogous to the equivalences $\PGL_2^{\times 3}/\PGL_2 \cong \RP^3 \times \RP^3$ and $\PSO_8/\G_2 \cong \RP^7 \times \RP^7$. If such a $G$ exists, and there is a good theory of $G\pw{t}$-equivariant sheaves of (``$2$-completed'') $k$-modules, it seems reasonable to expect that there is an equivalence of the form
    $$\Shv^{c,\Sat}_{G\pw{t}}(G\ls{t}/\DW_3\ls{t}; k) \cong \Perf^\sh(\std^{\otimes 3}(28,\vec{14},-14,\vec{0})/\SL_2(-14\rho)^{\times 3} \times \AA^2(8, 12)).$$
    Here, the ``Whittaker'' factor $\AA^2(8,12)$ on the right-hand side comes from the isomorphism
    $$\spf \H^\ast(B\DW_3; k) := \spf \H^\ast(B\DW_3; \Z_2) \otimes_{\Z_2} k \cong \widehat{\AA}^3(8, 12, 28),$$
    which follows from running the Bockstein spectral sequence on 
    $$\H^\ast(B\DW_3; \FF_2) \cong \FF_2\pw{w_8, w_{12}, w_{14}, w_{15}},$$
    and the fact that the Bockstein sends $w_{14} \mapsto w_{15}$.\footnote{A previous version of this remark asked for $G/\DW_3$ to be $\RP^{15} \times \RP^{15}$. One can check that such a $G$, if it existed, would have rational cohomology given by
    $$\H^\ast(BG; \Z_2) \otimes_{\Z_2} k \cong k\pw{c_4, c_6, c_{14}, x, y},$$
    where both $x$ and $y$ live in cohomological degree $16$.
    In an email, Jesper Grodal told me that such a $G$ cannot exist (it would have to be the $2$-completion of a compact Lie group, but no compact Lie group has the desired cohomology).} 
    Even if such a $G$ does not exist, one can still wonder about the analogue of the ``regular centralizer'' group scheme calculation from \cref{thm: main}: 
    \begin{itemize}
        \item Is there a good notion of \textit{genuine} {equivariant} $\DW_3$-cohomology (with coefficients in $k = \QQ_2(\zeta_8)$, say)? One should have $\spec \H^\ast_{\DW_3}(\ast; k)\cong \AA^3(8, 12, 28)$.
        \item Is there a faithful (basepoint-preserving) action of $\DW_3$ on $S^{15}$? Similarly, is there a faithful (basepoint-preserving) action of $\DW_3$ on the $2$-completion of a framed $30$-manifold $M^{30}$ with Kervaire invariant one?
        \item For the above expected action, is there an isomorphism
        $$\spec \H^{\DW_3}_\ast(\Omega M^{30}; k) \cong \AA^2(8, 12) \times (\AA^1(28) \times_{\std^{\otimes 3}(28,\vec{14},-14,\vec{0})/\SL_2(-14\rho)^{\times 3}} \AA^1(28))$$
        of graded group schemes over $k$?
    \end{itemize}
\end{remark}

\section{Variants}

\begin{remark}
    As in \cite{bhargava-composition-i}, understanding the $\SL_2^{\times 3}$-equivariant geometry of cubes can be specialized to understand variant situations. We will sketch some such variants below. 
    The observations motivating the discussion in this section are the main results of \cite{bhargava-composition-i}, and the analogy, likely already observed by the reader familiar with Bhargava's work, between the stack $B_{\AA^1\mmod (\Z/2)} \ld{J}$ and the narrow class group of a quadratic extension of $\QQ$. This analogy is not very surprising once one recognizes that $\ld{J}$ is just the nonsplit torus $x^2 - a^2 y^2 = 1$ over $\AA^1\mmod (\Z/2) = \spec \ringcoeff[a^2]$ (see \cite[Example 3.7.14]{ku-rel-langlands}).
\end{remark}
Let us begin with the case of squares/degree $2$ extensions. More precisely, recall from \cite{bf-derived-satake} that the (arithmetically sheared; see \cite{bzsv}) derived geometric Satake equivalence for $\PGL_2$ states:
\begin{theorem}[Bezrukavnikov-Finkelberg]\label{thm: bf derived satake}
    There is a monoidal equivalence of $\infty$-categories
    $$\Shv^{c,\Sat}_{\PGL_2^{\times 2}\pw{t}}(\PGL_2^{\times 2}\ls{t}/\PGL_2^\mathrm{diag}\ls{t}; \ringcoeff) \simeq \Perf^\sh(\sl_2^\ast(2-2\rho)/\SL_2(-2\rho)).$$
    The latter can alternatively be understood as the $\infty$-category of perfect complexes on a shearing of $T^\ast(\SL_2)/\SL_2^{\times 2}$.
\end{theorem}
As described in \cite[Section 3.2]{ku-rel-langlands}, one key input into \cref{thm: bf derived satake} is that if $\ld{J}$ denotes the group scheme of regular centralizers for $\SL_2$, with the embedding $\ld{J} \hookrightarrow \SL_2 \times \SL_2 \times \sl_2^\ast\mmod \SL_2$ via $g \mapsto (g, g^{-1})$, the affine closure of the quotient $(\SL_2 \times \SL_2 \times \sl_2^\ast\mmod \SL_2)/\ld{J}$ is isomorphic to $T^\ast \SL_2$.\footnote{The fact that the embedding sends $g\mapsto (g,g^{-1})$, as opposed to being the diagonal, is ultimately why the Chevalley involution shows up in the spectral side of geometric Langlands.}

One can ask for a \textit{variant} of \cref{thm: bf derived satake}, where the embedding $\PGL_2^\mathrm{diag}\ls{t} \hookrightarrow \PGL_2^{\times 2}\ls{t}$ is replaced by the embedding of $\PGL_2\ls{t}$ into $(\Res_{\cc\pw{t^{1/2}}/\cc\pw{t}} \PGL_2)\ls{t}$. The latter is the base-change to $\cc\ls{t}$ of the Weil restriction of the constant group scheme $\PGL_2$ along $\cc\pw{t} \subseteq \cc\pw{t^{1/2}}$. More generally, one could replace $\cc\pw{t^{1/2}}$ by $\cc\pw{t^{1/n}}$.  In this case, we have the following expectation.
\begin{conjecture}\label{conj: langlands nonsplit tn}
    Let $G_n = \Res_{\cc\pw{t^{1/n}}/\cc\pw{t}} G$.
    Then there is a fully faithful functor
    $$\Perf^\sh(\ld{G}\backslash \ol{(\ld{G} \times \ld{\g}^\ast(2)\mmod \ld{G})/\ld{J}[n]}) \hookrightarrow \Shv_{G_n\pw{t}}^{c}(G_n\ls{t}/G\ls{t}; \ringcoeff),$$
    where $\ld{G} \times \ld{\g}^\ast(2)\mmod \ld{G}$ denotes the constant group scheme over $\ld{\g}^\ast(2)\mmod \ld{G}$, $\ld{J}$ is the group scheme of regular centralizers for $\ld{G}$, $\ld{J}[n]$ is its $n$-torsion subgroup, and $\ol{(\ld{G} \times \ld{\g}^\ast(2)\mmod \ld{G})/\ld{J}[n]}$ is the affine closure of the quotient ${(\ld{G} \times \ld{\g}^\ast(2)\mmod \ld{G})/\ld{J}[n]}$.
    Furthermore, this full subcategory is stable under the action of the spherical Hecke category $\Shv_{(G_n \times G_n)\pw{t}}^{c}(G_n\ls{t}; \ringcoeff)$ on $\Shv_{G_n\pw{t}}^{c}(G_n\ls{t}/G\ls{t}; \ringcoeff)$.
\end{conjecture}
The affine closure $\ol{(\ld{G} \times \ld{\g}^\ast(2)\mmod \ld{G})/\ld{J}[n]}$ admits the structure of a (finite-type) Hamiltonian $\ld{G}$-scheme of dimension
$$\dim\left(\ol{(\ld{G} \times \ld{\g}^\ast(2)\mmod \ld{G})/\ld{J}[n]}\right) = \dim(\ld{G}) + \mathrm{rank}(\ld{G}) = 2\dim(\ld{G}/\ld{N}).$$
Moreover, its invariant-theoretic quotient by the $\ld{G}$-action is $\ld{\g}^\ast\mmod \ld{G}$. When $G = \PGL_2$, we describe this variety explicitly for $n=2,3,4$.

Although \cite[Theorem 3.6.4]{ku-rel-langlands} does not apply in this situation, because $G_n$ is not the base-change of a constant group scheme over $\cc$, we can nevertheless attempt to compute the analogue of the regular centralizer group scheme for the pair $G \hookrightarrow G_n$.
The following is a helpful tool in understanding these variant cases; assume for simplicity that $G$ is a connected reductive group whose derived subgroup is almost simple.
\begin{lemma}\label{lem: mult by n on homology}
    Let $n$ be a nonnegative integer, so that the degree $n$ map $S^2 \to S^2$ induces an $\E{1}$-endomorphism\footnote{This is an analogue for $\E{2}$-spaces of the observation that if $X$ is a space, the degree $n$ map $S^1 \to S^1$ induces a map $\Omega X \to \Omega X$ which sends $\gamma\to \gamma^n$. This is only a pointed map, and not necessarily an $\E{1}$-map, since taking powers is generally not a map of monoids. However, if $X$ is itself an $\E{1}$-space, so that $\Omega X$ is an $\E{2}$-space, the map $\gamma \to \gamma^n$ is one of $\E{1}$-spaces.} $[n]$ of $\Omega^2 BG \simeq \Omega G$. Under the homotopy equivalence $\Omega G \simeq G\ls{t}/G\pw{t}$, the map $[n]$ is induced by the map $\cc\pw{t} \to \cc\pw{t^{1/n}}$. Moreover, under the isomorphism 
    $$\spec \H^{G}_\ast(\Omega G; \ringcoeff) \cong \ld{J}$$
    of \cite[Theorem 1.2]{homology-langlands} (see also \cite{bfm}), the map $[n]^\ast: \co_{\ld{J}} \to \co_{\ld{J}}$ induced by $[n]: \Omega G \to \Omega G$ is given by multiplication by $n$ on the ring of functions $\co_{\ld{J}}$.
\end{lemma}
\begin{remark}
    Although the map $[n]: \Omega G \to \Omega G$ is only an $\E{1}$-endomorphism, it can be shown that the induced endomorphism of $C^{G}_\ast(\Omega G; \ringcoeff)$ is one of $\E{2}$-$\ringcoeff$-algebras. Upon Borel-completion (so $C^{G}_\ast(\Omega G; \ringcoeff)$ is replaced by $\ringcoeff[\Omega G]^{hG}$), this is a consequence of the observation that $\ringcoeff[\Omega G]^{hG}$ is an $\E{3}$-$\ringcoeff$-algebra: it is the $\E{2}$-Hochschild cohomology of $\ringcoeff[\Omega G]$, and the $\E{3}$-structure comes from the Deligne conjecture. Alternatively, the completion of $C^{G}_\ast(\Omega G; \ringcoeff)$ at the cellular filtration of $\Omega G$ can be identified with the $\E{2}$-Hochschild cohomology of $C_{G}^\ast(\ast; \ringcoeff)$, and the $\E{3}$-structure again comes from the Deligne conjecture. See \cite[Corollary 3.5.12]{ku-rel-langlands}.
\end{remark}
Let us now describe \cref{conj: langlands nonsplit tn} explicitly when $G = \PGL_2$.
\begin{example}[\cref{conj: langlands nonsplit tn} for $n=2$]
    \cref{lem: mult by n on homology} suggests that the analogue of the regular centralizer group scheme for \cref{conj: langlands nonsplit tn} is the the $2$-torsion subgroup $\ld{J}[2]$ of $\ld{J}$.
    As described in \cite{bfm}, $\ld{J}$ can be viewed as the group scheme over $\sl_2^\ast\mmod \SL_2 \cong \spec \ringcoeff[y]$ of matrices of the form $g = \begin{psmallmatrix}
        a & b \\
        by & a
    \end{psmallmatrix}$ with $\det(g) = a^2 - b^2y = 1$.
    To understand $\ld{J}[2]$, note that $g^2 = \begin{psmallmatrix}
        a^2 + b^2 y & 2ab \\
        2aby & a^2 + b^2 y
    \end{psmallmatrix}$, and since $2$ is a unit in $\ringcoeff$, we find that $g$ is $2$-torsion if and only if $ab = 0$ and $a^2 + b^2 y = 1$. But since $\det(g) = 1$, this forces $a^2 = 1$ and $b = 0$. In other words, $\ld{J}[2]$ is isomorphic to the constant group scheme $\mu_2 \times \sl_2^\ast\mmod \SL_2$. It follows that $(\SL_2 \times \sl_2^\ast\mmod \SL_2)/\ld{J}[2] \cong \PGL_2 \times \sl_2^\ast\mmod \SL_2$. No affine closure is necessary, since this is already affine. This suggests that there is a fully faithful functor from the $\infty$-category of perfect complexes on a shearing of the $\SL_2$-quotient of $\PGL_2 \times \sl_2^\ast\mmod \SL_2$ to $\Shv_{G\pw{t}}^{c}(G\ls{t}/\PGL_2\ls{t}; \ringcoeff)$, i.e., a fully faithful functor
    \begin{align*}
        \Perf^\sh((\PGL_2 \times \sl_2^\ast(2)\mmod \SL_2)/\SL_2) & \simeq \Perf^\sh(B\mu_2 \times \sl_2^\ast(2)\mmod \SL_2) \\
        & \hookrightarrow \Shv_{G\pw{t}}^{c}(G\ls{t}/\PGL_2\ls{t}; \ringcoeff).
    \end{align*}
    This is the specialization of \cref{conj: langlands nonsplit tn} to the present case. Such a fully faithful functor does indeed exist: the inclusion of a basepoint of each of the two connected components of $G\ls{t}$ corresponds to the inclusion of each of the two factors of $\Perf^\sh(B\mu_2 \times \sl_2^\ast(2)\mmod \SL_2) \simeq \Perf^\sh(\sl_2^\ast(2)\mmod \SL_2)^{\oplus 2}$.
\end{example}

Let us now turn to the cubic case. 
\begin{conjecture}[\cref{conj: langlands nonsplit tn} for $n=3$]\label{conj: langlands binary cubics}
    Let $\Sym^3(\std)(4,2,0,-2)$ denote the graded vector space of binary cubic forms, where such a form is viewed as a function $\AA^2(-2,0) \to \AA^1(-2)$. In other words, the coefficients of $ax^3 + 3bx^2 y + 3cxy^2 + dy^3$ have the following weights: $a$ lives in weight $-4$, $b$ lives in weight $-2$, $c$ lives in weight $0$, and $d$ lives in weight $2$.
    Let $G_3 = \Res_{\cc\pw{t^{1/3}}/\cc\pw{t}} \PGL_2$.
    Then there is a fully faithful functor
    $$\Perf^\sh(\Sym^3(\std)(4,2,0,-2)/\SL_2(-2\rho)) \hookrightarrow \Shv_{G_3\pw{t}}^{c}(G_3\ls{t}/\PGL_2\ls{t}; \ringcoeff).$$
\end{conjecture}
Note that the example of $\SL_2$ acting on $\Sym^3(\std)$ does \textit{not} fit into the formalism of \cite{bzsv}, since it is not hyperspherical in the sense of \textit{loc. cit.} (see \cite[Example 5.1.10]{bzsv}).
Let us show that \cref{conj: langlands binary cubics} is indeed \cref{conj: langlands nonsplit tn} specialized to $n=3$.
\begin{prop}\label{prop: regular locus binary cubic}
    Let $V = \Sym^3(\std)$ denote the $4$-dimensional symplectic vector space of binary cubic forms, so that $V$ admits an action of $\SL_2$. Then:
    \begin{enumerate}
        \item Let $\Delta: V \to \AA^1$ denote the map sending a binary cubic form $f = ax^3 + 3bx^2 y + 3cxy^2 + dy^3$ to its discriminant
        $$\Delta(f) = a^2 d^2 - 6abcd - 3b^2c^2 + 4 (ac^3 + b^3d).$$
        Then $\Delta$ defines an isomorphism $V\mmod \SL_2 \cong \AA^1$.
        \item The closed immersion $\kappa: \AA^1 \to V$ sending $a \mapsto -\frac{a}{4} x^3 + 3xy^2$ defines a section of $\Delta$, and the $\SL_2$-orbit of the image of $\kappa$ has complement of codimension $\geq 2$.
        \item Identify $\AA^1 = \sl_2^\ast\mmod \SL_2$, let $\ld{J}$ denote the group scheme over $\sl_2^\ast\mmod \SL_2$ of regular centralizers for $\SL_2$, and let $\ld{J}[3]$ denote its $3$-torsion subgroup. Then there is an isomorphism
        $$\sl_2^\ast\mmod \SL_2 \times_{V/\SL_2} \sl_2^\ast\mmod \SL_2 \cong \ld{J}[3]$$
        of group schemes over $\sl_2^\ast\mmod \SL_2$.
        In particular, the affine closure of $(\SL_2 \times \sl_2^\ast\mmod \SL_2)/\ld{J}[3]$ is $\SL_2$-equivariantly isomorphic to $V$.
    \end{enumerate}
\end{prop}
\begin{proof}
    The first statement is in \cite[Section 0.12]{vinberg-popov-invariant}, and the second statement can be deduced similarly.
    For the final statement, recall as in \cite{bhargava-composition-i} that there is a closed immersion $V \subseteq \std^{\otimes 3}$ given by $ax^3 + 3bx^2 y + 3cxy^2 + dy^3 \mapsto (a, \vec{b}, d, \vec{c})$. This corresponds to the triply-symmetric cube
    $$\xymatrix@=.75em{
    & b \ar@{-}[rr] \ar@{-}'[d][dd] && c \ar@{-}[dd] \\
    a \ar@{-}[ru] \ar@{-}[rr] \ar@{-}[dd] && b \ar@{-}[ru] \ar@{-}[dd] & \\
    & c \ar@{-}'[r][rr] && d. \\
    b \ar@{-}[ru] \ar@{-}[rr] && c \ar@{-}[ru] &
    }$$
    The above embedding is $\SL_2$-equivariant for the natural action on $V$ and the diagonally embedded $\SL_2^\mathrm{diag} \subseteq \SL_2^{\times 3} = \ld{G}$ acting on $\std^{\otimes 3}$. Moreover, the composite
    $$V \subseteq \std^{\otimes 3} \xar{\det} \sl_2^\ast\mmod \SL_2 \cong \AA^1$$
    sends $f\mapsto \Delta(f)$. This implies that $\sl_2^\ast\mmod \SL_2 \times_{V/\SL_2} \sl_2^\ast\mmod \SL_2$ can be identified with the intersection $\sl_2^\ast\mmod \SL_2 \times_{\std^{\otimes 3}/\ld{G}} \sl_2^\ast\mmod \SL_2$ with the diagonally embedded $\SL_2^\mathrm{diag} \times \sl_2^\ast\mmod \SL_2 \subseteq \SL_2^{\times 3} \times \sl_2^\ast\mmod \SL_2$. By \cref{prop: stab of kostant}, we find that
    \begin{align*}
        \sl_2^\ast\mmod \SL_2 \times_{V/\SL_2} \sl_2^\ast\mmod \SL_2 & \cong \ker(\ld{J} \times_{\sl_2^\ast\mmod \SL_2} \ld{J} \times_{\sl_2^\ast\mmod \SL_2} \ld{J} \xar{\mathrm{prod}} \ld{J}) \cap (\SL_2^\mathrm{diag} \times \sl_2^\ast\mmod \SL_2) \\
        & \cong \ld{J}[3].
    \end{align*}
    The claim about the affine closure of $(\SL_2 \times \sl_2^\ast\mmod \SL_2)/\ld{J}[3]$ follows from (b).
\end{proof}
\begin{remark}
    If we identify $\Sym^3(\std)$ with the affine closure of $(\SL_2 \times \sl_2^\ast\mmod \SL_2)/\ld{J}[3]$ via \cref{prop: regular locus binary cubic}, the inclusion $\ld{J}[3] \subseteq \ld{J}$ along with the identification of $\sl_2^\ast$ with the affine closure of $(\SL_2 \times \sl_2^\ast\mmod \SL_2)/\ld{J}$ gives an $\SL_2$-equivariant map $\Sym^3(\std) \to \sl_2^\ast \cong \Sym^2(\std)$. (Note that this map is quadratic, not linear.) This map is very classical: it produces the \textit{quadratic resolvent} of a binary cubic form. Explicitly, it sends $f = ax^3 + 3bx^2 y + 3cxy^2 + dy^3$ to the binary quadratic form
    $$q = (ac - b^2) x^2 + (ad - bc) xy + (bd - c^2) y^2.$$
    Said differently, $\Sym^3(\std)$ admits the structure of a Hamiltonian $\SL_2$-space, and the moment map $\Sym^3(\std) \to \sl_2^\ast$ for this $\SL_2$-action can be identified with the quadratic resolvent construction. Moreover, this moment map is $\SL_2(-2\rho) \rtimes \GG_m$-equivariant for the grading on $\Sym^3(\std)$ described in \cref{conj: langlands binary cubics} and the $(2-2\rho)$-grading on $\sl_2^\ast$.

    There is an action of $\Shv_{G_3\pw{t} \times G_3\pw{t}}^{c,\Sat}(G_3\ls{t}; \ringcoeff)$ on $\Shv_{G_3\pw{t}}^{c}(G_3\ls{t}/\PGL_2\ls{t}; \ringcoeff)$ by convolution. \cref{thm: bf derived satake} allows us to identify $\Shv_{G_3\pw{t} \times G_3\pw{t}}^{c,\Sat}(G_3\ls{t}; \ringcoeff)$ with $\Perf^\sh(\sl_2^\ast(2-2\rho)/\SL_2(-2\rho))$ as monoidal categories. Under \cref{conj: langlands binary cubics}, the resulting action of $\Perf^\sh(\sl_2^\ast(2-2\rho)/\SL_2(-2\rho))$ on $\Shv_{G_3\pw{t}}^{c}(G_3\ls{t}/\PGL_2\ls{t}; \ringcoeff)$ should preserve the full subcategory $\Perf^\sh(\Sym^3(\std)(4,2,0,-2)/\SL_2(-2\rho))$. The action on this full subcategory should be given by the moment map/quadratic resolvent $\Sym^3(\std) \to \sl_2^\ast$.
\end{remark}
\begin{remark}\label{rmk: binary quartics}
    What about binary forms of higher degree? Let $V$ denote the $5$-dimensional affine space of binary quartic forms; the action of $\SL_2$ on $V$ descends to an action of $\PGL_2$.
    \begin{enumerate}
        \item Let $\pi: V \to \AA^2$ denote the map sending a binary quartic form $f = ax^4 + 4bx^3 y + 6cx^2y^2 + 4dxy^3 + ey^4$ to the invariants
        \begin{align*}
            I & = ae - 4bd + 3c^2, \\
            J & = ace + 2 bcd - ad^2 - b^2e - c^3.
        \end{align*}
        Then $\pi$ defines an isomorphism $V\mmod \PGL_2 \cong \AA^2$. 
        \item The closed immersion $\kappa: \AA^2 \to V$ sending $(a,b) \mapsto 4x^3 y + dxy^3 + ey^4$ defines a section of $\pi$, and the $\PGL_2$-orbit of the image of $\kappa$ has complement of codimension $\geq 2$. In fact, the $\PGL_2$-orbit consists of those binary quartic forms with at least one root of multiplicity $1$.
        \item Let $\ce$ denote the elliptic curve over $\AA^2 = \spec \ringcoeff[d, e]$ given by $y^2 = x^3 + dx + e$, and let $\ce[2]$ denote its $2$-torsion subgroup. Then there is an isomorphism
        $$\AA^2 \times_{V/\PGL_2} \AA^2 \cong \ce[2]$$
        of group schemes over $\AA^2$.
        In particular, the affine closure of $(\PGL_2 \times \AA^2)/\ce[2]$ is $\PGL_2$-equivariantly isomorphic to $V$. (Tantalizingly, the coordinates $d$ and $e$ live in weights $4$ and $6$, and so the base $\AA^2$ can actually be identified with $\spec \H^\ast_{\SL_3}(\ast; \ringcoeff)$.)
    \end{enumerate}
    Once one knows the formulas, (a) and (b) are not difficult calculations (see \cite{vinberg-popov-invariant}), and part (c) can be proved as in \cite[Sections 3-5]{cremona-fisher} and \cite[Theorem 3.2]{bhargava-shankar}.
    Finally, note that if $V = \Sym^j(\AA^2)$ denote the $(j+1)$-dimensional affine space of binary $j$-forms, so that $V$ admits an action of $\SL_2$, the invariant-theoretic quotient $V\mmod \SL_2$ is \textit{not} an affine space if $j\geq 5$ (see \cite[Example 1 in Section 8.2]{vinberg-popov-invariant}).
\end{remark}
The poset of $\SL_2$-orbit closures in $\PP(\Sym^3(\std))$ is shown in \cref{fig: orbits Sym3}.
\begin{figure}[H]
\adjustbox{scale=1,center}{%
\begin{tikzcd}
{\PP(\Sym^3(\std))} & {\{\Delta = 0\}} & {\PP^1}
\arrow[no head, from=1-1, to=1-2]
\arrow[no head, from=1-2, to=1-3]
\end{tikzcd}
}
\captionsetup{width=\linewidth}
\caption[]{$\SL_2$-orbit closures on $\Sym^3(\std)$, connected by closure. The inclusion $\PP^1 \hookrightarrow \PP^3$ is the embedding of the twisted cubic, and the vanishing locus of $\Delta$ is the dual variety/tangent developable of the twisted cubic.}
\label{fig: orbits Sym3}
\end{figure}

Before proceeding to \cref{conj: langlands nonsplit tn} for higher $n$, let us note the following modification of the cubic case.
Instead of considering $\PGL_2\pw{t}$ as a subgroup of $\Res_{\cc\pw{t^{1/3}}/\cc\pw{t}} \PGL_2$, one could also consider the Weil restriction $\Res_{\cc\pw{t^{1/2}} \times \cc\pw{t}/\cc\pw{t}} \PGL_2$ along the extension $\cc\pw{t} \subseteq \cc\pw{t^{1/2}} \times \cc\pw{t}$. For this, one obtains the following variant of \cref{thm: main}.
\begin{conjecture}\label{conj: langlands pairs of binary quadratic}
    Let $(\std \otimes \sl_2^\ast)(4,2,0,-2)$ denote the graded vector space of pairs of binary quadratic forms, where the coefficients of a pair $(q_1, q_2) = (ax^2 + 2bxy + cy^2, dx^2 + 2exy + fy^2)$ have the following weights: $a$ lives in weight $-4$, $b$ lives in weight $-2$, $c$ lives in weight $0$, $d$ lives in weight $-2$, $e$ lives in weight $0$, and $f$ lives in weight $2$. In other words, $(\std \otimes \sl_2^\ast)(4,2,0,-2) \cong \AA^2(2,0) \otimes \sl_2^\ast(-2\rho)$.
    Let $G = \Res_{\cc\pw{t^{1/2}} \times \cc\pw{t}/\cc\pw{t}} \PGL_2$. 
    Then there is a fully faithful functor
    $$ \Perf^\sh((\std \otimes \sl_2^\ast)(4,2,0,-2)/\SL_2(-2\rho)^{\times 2}) \hookrightarrow \Shv_{G\pw{t}}^{c}(G\ls{t}/\PGL_2\ls{t}; \ringcoeff).$$
\end{conjecture}
Again, \cite[Theorem 3.6.4]{ku-rel-langlands} does not apply in this situation, because $G$ is not the base-change to $\cc\ls{t}$ of a constant group scheme over $\cc$. Nevertheless, we expect that \cref{conj: langlands binary cubics} is a consequence of \cref{lem: mult by n on homology} and \cref{prop: reg centr pairs of binary quadratic} below; together, these results should give an analogue of the criteria of \cite[Theorem 3.6.4]{ku-rel-langlands}.
\begin{prop}\label{prop: reg centr pairs of binary quadratic}
    Let $V = \std \otimes \sl_2^\ast$, equipped with an action of $\SL_2 \times \SL_2$ via the $\SL_2$-actions on $\std$ and $\sl_2^\ast$. Then:
    \begin{enumerate}
        \item Let $\Delta: V \to \AA^1$ denote the map sending a pair of binary quadratic forms $(q_1, q_2) = (ax^2 + 2bxy + cy^2, dx^2 + 2exy + fy^2)$ to the function
        $$\Delta(q_1, q_2) = a^2 f^2 + c^2 d^2 - 2 acdf + 4 (ae - bd)(ce - bf).$$
        Then $\Delta$ defines an isomorphism $V\mmod (\SL_2 \times \SL_2) \cong \AA^1$.
        \item The closed immersion $\kappa: \AA^1 \to V$ sending $a \mapsto (\frac{a}{4}x^2 + y^2, 2xy)$ defines a section of $\Delta$, and the $\SL_2 \times \SL_2$-orbit of the image of $\kappa$ has complement of codimension $\geq 2$.
        \item Identify $\AA^1 = \sl_2^\ast\mmod \SL_2$, let $\ld{J}$ denote the group scheme over $\sl_2^\ast\mmod \SL_2$ of regular centralizers for $\SL_2$, and define the embedding
        $$\ld{J} \hookrightarrow \SL_2 \times \SL_2 \times \sl_2^\ast\mmod \SL_2, \ g \mapsto (g^{-2}, g).$$
        Note that this is indeed a homomorphism since $\ld{J}$ is \textit{commutative}.
        Then there is an isomorphism
        $$\sl_2^\ast\mmod \SL_2 \times_{V/(\SL_2 \times \SL_2)} \sl_2^\ast\mmod \SL_2 \cong \ld{J} \subseteq \SL_2 \times \SL_2 \times \sl_2^\ast\mmod \SL_2$$
        of group schemes over $\sl_2^\ast\mmod \SL_2$.
        In particular, the affine closure of $(\SL_2 \times \SL_2 \times \sl_2^\ast\mmod \SL_2)/\ld{J}$ is $\SL_2 \times \SL_2$-equivariantly isomorphic to $V$.
    \end{enumerate}
\end{prop}
\begin{proof}[Proof sketch]
    These statements follow exactly as in \cref{prop: regular locus binary cubic}. Indeed, as described in \cite{bhargava-composition-i}, there is a closed immersion $V \subseteq \std^{\otimes 3}$ given by 
    $$(ax^2 + 2bxy + cy^2, dx^2 + 2exy + fy^2) \mapsto (a, (b, d, b), f, (e,c,e)).$$
    This corresponds to the doubly-symmetric cube
    $$
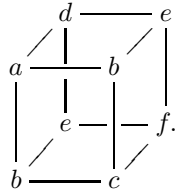
\xymatrix@=.75em{
    & d \ar@{-}[rr] \ar@{-}'[d][dd] && e \ar@{-}[dd] \\
    a \ar@{-}[ru] \ar@{-}[rr] \ar@{-}[dd] && b \ar@{-}[ru] \ar@{-}[dd] & \\
    & e \ar@{-}'[r][rr] && f. \\
    b \ar@{-}[ru] \ar@{-}[rr] && c \ar@{-}[ru] &
    }$$
    The above morphism is $\SL_2 \times \SL_2$-equivariant for the natural action on $V$ and the embedding 
    $$\iota: \SL_2 \times \SL_2 \subseteq \SL_2^{\times 3} = \ld{G}, \ (g, h) \mapsto (g, h, h).$$
    One can check that the composite
    $$V \subseteq \std^{\otimes 3} \xar{\det} \sl_2^\ast\mmod \SL_2 \cong \AA^1$$
    sends $(q_1, q_2) \mapsto \Delta(q_1, q_2)$. Finally, as in \cref{prop: regular locus binary cubic}, we find that
    \begin{align*}
        \hspace*{-1cm}
        \sl_2^\ast\mmod \SL_2 \times_{V/(\SL_2 \times \SL_2)} \sl_2^\ast\mmod \SL_2 & \cong \ker(\ld{J} \times_{\sl_2^\ast\mmod \SL_2} \ld{J} \times_{\sl_2^\ast\mmod \SL_2} \ld{J} \xar{\mathrm{prod}} \ld{J}) \cap (\iota(\SL_2 \times \SL_2) \times \sl_2^\ast\mmod \SL_2) \\
        & \cong \ker(\ld{J} \times_{\sl_2^\ast\mmod \SL_2} \ld{J} \xar{(g,h)\mapsto gh^2} \ld{J}) \\
        & \cong \ld{J},
    \end{align*}
    where $\ld{J}$ is a subgroup of $(\iota(\SL_2 \times \SL_2) \times \sl_2^\ast\mmod \SL_2)$ via $g \mapsto (g^{-2}, g)$, as desired.
\end{proof}
\begin{remark}
    The inclusion $\ld{J} \hookrightarrow \SL_2 \times \SL_2 \times \sl_2^\ast\mmod \SL_2$ from \cref{prop: reg centr pairs of binary quadratic}(c) can be alternatively described as the composite inclusion
    $$\ld{J} \cong \ker(\ld{J} \times_{\sl_2^\ast\mmod \SL_2} \ld{J} \xar{\id \times [2]^\ast} \ld{J}) \hookrightarrow \ld{J} \times_{\sl_2^\ast\mmod \SL_2} \ld{J} \hookrightarrow \SL_2 \times \SL_2 \times \sl_2^\ast\mmod \SL_2.$$
\end{remark}
    The poset of $\SL_2 \times \SL_2$-orbit closures in $\PP(\std \otimes \sl_2^\ast)$ is shown in \cref{fig: orbits std tensor sl2}.
    \begin{figure}[H]
    \adjustbox{scale=1,center}{%
    \begin{tikzcd}
	{\PP(\std \otimes \sl_2^\ast)} & {\{\Delta = 0\}} & {\PP^1 \times \PP^2} & {\PP^1 \times \PP^1}
	\arrow[no head, from=1-1, to=1-2]
	\arrow[no head, from=1-2, to=1-3]
	\arrow[no head, from=1-3, to=1-4]
    \end{tikzcd}
    }
    \captionsetup{width=\linewidth}
    \caption[]{$\SL_2 \times \SL_2$-orbit closures on $\PP(\std \otimes \sl_2^\ast)$, connected by closure. The inclusion $\PP^1 \times \PP^2 \hookrightarrow \PP^5$ is the Segre embedding, the embedding $\PP^1 \times \PP^1 \hookrightarrow \PP^1 \times \PP^2$ is induced by the inclusion of a quadric $\PP^1 \subseteq \PP^2$, and the vanishing locus of $\Delta$ is the dual variety to this quadric.}
    \label{fig: orbits std tensor sl2}
    \end{figure}
\begin{remark}\label{rmk: pairs of binary cubics}
    There is a variant of \cref{prop: reg centr pairs of binary quadratic} for the vector space $V = \std \otimes \Sym^3(\std)$ of pairs of binary \textit{cubic} forms, equipped with its natural $\SL_2 \times \SL_2$-action. We will not use this, so we will just state the relevant results for the sake of completeness (these statements follow from the work of Bhargava-Ho in \cite{bhargava-ho}). It turns out that there is an isomorphism $V\mmod (\SL_2 \times \SL_2) \cong \AA^2$, where the invariants, denoted $a_1$ and $a_3$, are of degrees $2$ and $6$, respectively. In particular, there is an isomorphism 
    $$V(2)\mmod (\SL_2 \times \SL_2) \cong \AA^2(4, 12),$$
    where the symbol $V(2)$ means that the coordinates of $V$ are placed in weight $-2$. In fact, the action of $\SL_2 \times \SL_2$ on $V$ descends to an action of $\SO_4$, and it turns out that $V \cong \so_4^\perp \subseteq \g_2^\ast$ under the embedding $\SO_4 \subseteq \G_2$. (Tantalizingly, $\H^\ast_{\G_2}(\ast; \QQ)$ is a polynomial algebra on classes in cohomological degrees $4$ and $12$, and so the base $\AA^2(4,12)$ can actually be identified with $\spec \H^\ast_{\G_2}(\ast; \ringcoeff)$.)
    
    As in \cref{rmk: binary quartics}, there is a Kostant slice $\kappa: \AA^2 \cong V\mmod (\SL_2 \times \SL_2) \to V$, and it turns out that there is an isomorphism
    $$\AA^2 \times_{V/(\SL_2 \times \SL_2)} \AA^2 \cong \ce_{\Gamma_1(3)}[2],$$
    where $\ce_{\Gamma_1(3)}$ is the universal elliptic curve $y^2 + a_1 xy + a_3 y = x^3$ with a $\Gamma_1(3)$-level structure over $\AA^2$.
    The reader is referred to \cite{bhargava-ho} (see in particular \cite[Line 6 of Table 1]{bhargava-ho}) for a detailed study of this example, from which the above claims can be deduced. I do not know how this example fits into the invariant-theoretic picture of geometric Langlands described in this article and \cite{bzsv, ku-rel-langlands}!
\end{remark}

Finally, using \cref{rmk: so8 minimal nilpotent}, here is the specialization of \cref{conj: langlands nonsplit tn} to $n=4$:
\begin{conjecture}[\cref{conj: langlands nonsplit tn} for $n=4$]\label{conj: langlands nonsplit t4}
    Let $G_4 = \Res_{\cc\pw{t^{1/4}}/\cc\pw{t}} \PGL_2$. Let $\ol{\co}$ denote the minimal nilpotent coadjoint orbit of $\fr{so}_8$, and let the symmetric group $\Sigma_4$ act on $\ol{\co}$ as described in \cite[Section 4.1]{S4-action-SO8}. Then the action of $\SL_2^{\times 4} \subseteq \SO_8$ on $\ol{\co}$ descends to an action of $\SL_2$ on the invariant quotient $\ol{\co}\mmod\Sigma_4$.
    For this action of $\SL_2$, there is a fully faithful functor
    $$\Perf^\sh((\ol{\co}\mmod\Sigma_4)/\SL_2(-2\rho)) \hookrightarrow \Shv_{G_4\pw{t}}^{c}(G_4\ls{t}/\PGL_2\ls{t}; \ringcoeff),$$
    for a certain grading on $\ol{\co}\mmod\Sigma_4$.
\end{conjecture}
Let us end with a comment about $\ld{J}[n]$ for arbitrary $n$.
\begin{remark}
    One could try to generalize \cref{prop: regular locus binary cubic} to binary forms of higher degree.
    Let $n\geq 3$, and let $V = \Sym^n(\std)$ denote the $(n+1)$-dimensional vector space of binary forms of degree $n$, so that $V$ admits an action of $\SL_2$. Let $\kappa: \AA^1\mmod (\Z/2) = \spec \ringcoeff[a^2] \to V$ denote the closed immersion sending
    $$a^2\mapsto \begin{cases}
        \frac{1}{2} ((y + a x)^n + (y - a x)^n) & n \text{ even}; \\
        \frac{1}{2a} ((y + a x)^n - (y - a x)^n) & n \text{ odd}.
    \end{cases}$$
    A direct computation then shows that the stabilizer \textit{inside $\ld{J}$} of $\kappa$ is isomorphic to $\ld{J}[n]$. Using the $\SL_2$-action on $V$, one therefore obtains a map 
    $$f: \ol{(\SL_2 \times \AA^1)/\ld{J}[n]} \to \Sym^n(\std).$$
    When $n$ is odd, the vector space $\Sym^n(\std)$ admits a symplectic structure: if we write $\sum_{j=0}^n \binom{n}{j} a_j x^j y^{n-j}$ for a general binary form of degree $n$, the symplectic form (up to scaling) is given by 
    $$\omega = \sum_{j=0}^{(n-1)/2} (-1)^j \binom{n}{j} da_j \wedge da_{n-j}.$$
    Using this formula, one can check explicitly that the following diagram commutes:
    \begin{equation}\label{eq: moment-factorization Sym-odd}
        \xymatrix{
        \ol{(\SL_2 \times \AA^1)/\ld{J}[n]} \ar[r]^-f \ar[d] & \Sym^n(\std) \ar[d]^-\mu \\
        \ol{(\SL_2 \times \AA^1)/\ld{J}} \ar[r]_-\sim & \sl_2^\ast.
        }
    \end{equation}
    here, the vertical map is the moment map for $\Sym^n(\std)$.
    When $n=3$, the map $f$ is an isomorphism by \cref{prop: regular locus binary cubic}, but this will not be true for $n\geq 4$. Indeed, $\kappa$ is generally not a section of an invariant-theoretic quotient map (see \cref{rmk: binary quartics})!

    The above discussion has some consequences for \textit{global} function fields.
    Let $\Sigma$ denote a smooth projective curve over $\cc$ of genus $g$.
    As described in \cite{ginzburg-rozenblyum}, any symplectic $\ld{G}$-representation $V$ and a choice of square root $K_\Sigma^{1/2}$ of the canonical bundle of $\Sigma$ defines a Lagrangian substack $\cL_V \hookrightarrow \Higgs_{\ld{G}}(\Sigma)$ of the moduli stack of $\ld{G}$-Higgs bundles on $\Sigma$. (One could more generally replace $V$ by a graded Hamiltonian $\ld{G}$-variety.) Roughly speaking, this Lagrangian substack consists of those Higgs bundles $(\cP, \theta)$ with $\theta \in \H^0(\Sigma; \ld{\g}^\ast_\cP \otimes K_\Sigma)$ obtained as the moment map image of a section of $V_\cP \otimes K_\Sigma^{1/2}$.
    
    \cref{prop: regular locus binary cubic} and the above discussion together allow us to describe this Lagrangian explicitly when $\ld{G} = \SL_2$ and $V = \Sym^{2n+1}(\std)$. Namely, the Lagrangian $\cL_{\Sym^{2n+1}(\std)} \to \Higgs_{\SL_2}(\Sigma)$ fiberwise (over the Hitchin base) contains the $(2n+1)$-torsion subgroup in the Prym variety; when $n=1$, it is furthermore \textit{exactly} the $3$-torsion in the Prym variety. One can obtain similar descriptions for the other symplectic representations described in this article. For instance, \cref{prop: stab of kostant} implies that if $\ld{G} = \SL_2^{\times 3}$ and $V = \std^{\otimes 3}$, the Lagrangian $\cL_{\std^{\otimes 3}} \to \Higgs_{\SL_2^{\times 3}}(\Sigma)$ is fiberwise (over the Hitchin base) given by the subgroup of the threefold product of the Prym variety consisting of those triples $(L_1, L_2, L_3)$ of line bundles whose tensor product is trivial.
\end{remark}

\bibliographystyle{alphanum}
\bibliography{main}
\end{document}